\documentclass[11pt, preprint]{article}

\usepackage[letterpaper,margin=1.1in]{geometry}
\usepackage{amsmath,amssymb,amsthm,mathtools}
\usepackage{graphicx}
\usepackage{subcaption}
\usepackage{algorithm}
\usepackage{algpseudocode}
\usepackage{booktabs,tabularx,array}

\usepackage{float}
\usepackage{placeins}
\usepackage{xcolor}
\usepackage{listings}
\usepackage[hidelinks]{hyperref}
\usepackage{microtype}
\usepackage{biblatex}
\usepackage{todonotes}

\title{A Geometric View of Adaptive Cross Approximation via Exterior Algebra}%7E
\author{Trevor Loe\thanks{UCLA, Los Angeles, CA (\href{mailto:tgloe@math.ucla.edu}{tgloe@math.ucla.edu})}, Longxiu Huang\thanks{Michigan State, East Lansing, MI (\href{mailto:huangl3@msu.edu}{huangl3@msu.edu})}, Deanna Needell\thanks{UCLA, Los Angeles, CA and Univ. British Columbia, Vancouver, BC (\href{mailto:deanna@math.ucla.edu}{deanna@math.ucla.edu})}}
\date{}

\newtheorem{theorem}{Theorem}[section]
\newtheorem{prop}[theorem]{Proposition}
\newtheorem{lemma}[theorem]{Lemma}
\newtheorem{corollary}[theorem]{Corollary}

\newcolumntype{Y}{>{\raggedright\arraybackslash}X}
\newcolumntype{L}[1]{>{\raggedright\arraybackslash}p{#1}}

\newcommand{\rank}{\operatorname{rank}}
\newcommand{\R}{\mathbb{R}}
\newcommand{\inner}[1]{\langle #1\rangle}
\newcommand{\prob}{\mathbb{P}}
\newcommand{\Exp}{\mathbb{E}}

\newcommand{\Span}{\operatorname{span}}
\renewcommand{\O}{\mathcal{O}}

\newcommand{\algorithmicbreak}{\textbf{break}}

\newcommand{\safeincludegraphics}[2][]{%
  \IfFileExists{#2}{\includegraphics[#1]{#2}}{%
    \fbox{\parbox[c][2.0in][c]{0.82\linewidth}{\centering
      Figure file not included in the supplied draft:\\[0.4em]
      {\footnotesize\nolinkurl{#2}}}}}}

\providecommand{\keywords}[1]
{
  \small	
  \textbf{Keywords. } #1
}

\begin{document}

\maketitle

\begin{abstract}
Adaptive cross approximation (ACA) constructs low-rank CUR approximations from selected rows and columns of a matrix, making it attractive when individual entries are inexpensive to query but forming or repeatedly multiplying by the full matrix is not. Its practical performance, however, is not well explained by existing worst-case bounds, which can be exponentially large and often fail to predict when greedy largest-entry pivoting performs poorly. By viewing the matrix as a cross-gram matrix, we reinterpret the classical determinantal formula for the CUR residual through the lens of exterior algebra. In this representation, each residual entry is a weighted inner product between two $(k+1)$-blades divided by the normalized volume of the pivot block. The formula separates three effects: the weighted magnitudes of singular-vector rows, their angular geometry, and the conditioning of the selected pivots. From this we recover a classical $\sigma_{k+1}$-type estimate, characterize when one pivot annihilates or nearly annihilates additional rows and columns, obtain a quasi-optimality estimate for positive-semidefinite matrices, and bound the rank-one update through an explicit angle-based formula. We use this framework to explain the failure of greedy ACA on an asymmetric spiral-kernel example and to motivate a geometry-aware weighted-mass pivoting rule for radial-basis Galerkin stiffness matrices. In the reported experiments, weighted-mass pivoting often reduces the Frobenius residual relative to greedy pivoting and randomly pivoted Cholesky while retaining comparable runtime on parallel hardware.
\end{abstract}


\section{Introduction}

Many matrices arising in scientific computing, inverse problems, and data
analysis have rapidly decaying singular values and therefore admit accurate
low-rank approximations. Examples include Galerkin boundary-element matrices
\cite{Bebendorf_AcceleratingGalerkinBEM2006,
Bebendorf_ApproximationBoundaryElement2000}, collocation matrices
\cite{Bebendorf_AdaptiveLowRankApproximation2003}, approximate solution
operators for advection equations
\cite{Zheng_SemiLagrangianAdaptiveRankSLAR2024}, and kernel matrices associated
with smooth positive-definite kernels
\cite{Chen_RandomlyPivotedCholesky2025}. Low-rank representations can reduce
storage, accelerate matrix operations, and facilitate the solution of
large-scale linear systems.

For a matrix $A\in\R^{n\times m}$ with singular values
$\sigma_1\geq \sigma_2\geq\cdots\geq \sigma_{\min\{m,n\}}$, the Eckart--Young--Mirsky theorem identifies
the truncated singular value decomposition as the best rank-$k$ approximation
in both the operator and Frobenius norms. The corresponding errors are
$\sigma_{k+1}$ and
\(
    \left(\sum_{\ell>k}\sigma_\ell^2\right)^{1/2},
\)
respectively. The SVD is therefore the natural approximation benchmark.
However, its factors are generally dense linear combinations of the rows and
columns of $A$, and computing them requires global access to the matrix. Furthermore, computing a full SVD generally has complexity $\O(n^3)$, which is prohibitively expensive for large $n$. 
Randomized SVD and subspace-iteration methods reduce the computational cost,
but they still rely on matrix--vector or matrix--matrix products
\cite{Tropp_RandomizedAlgorithmsLowrank2023}. In applications where selected
entries, rows, or columns can be evaluated inexpensively but repeated global
matrix operations are unavailable or costly, an entrywise approximation
method is especially attractive.

\subsection{CUR decompositions}

A CUR decomposition approximates $A$ using actual columns and rows of the
matrix. Given index sets $I\subseteq[n]$ and $J\subseteq[m]$, define
\(
    C=A(:,J),R=A(I,:).
\)
A general CUR approximation takes the form
\(
    A\approx C WR,
\)
where the core matrix $W$ is constructed from the selected submatrices. When the selected intersection has the appropriate rank, exact CUR decompositions admit several equivalent formulations \cite{Hamm_PerspectivesCUR2020}. In particular, if
\( \rank(A(I,J))=\rank(A)=r,\)
then
\(
    A=A(:,J)A(I,J)^\dagger A(I,:)
\), where $A^\dagger$ represents the pseudoinverse of $A$. 
This work focuses on the square, nonsingular-intersection case
\[
    A\approx A(:,J)A(I,J)^{-1}A(I,:),
\]
which is commonly called a cross, skeleton, or pseudoskeleton approximation.

Because $C$ and $R$ consist of actual columns and rows of the input matrix,
CUR decompositions can retain interpretability, sparsity, localization,
nonnegativity, and other structure that may be obscured by singular-vector
bases. This feature has motivated CUR methods for interpretable data analysis
and feature selection
\cite{Drineas_RelativeErrorCUR2008,Mahoney_CURData2009}, as well as for
subspace clustering, where selected columns and rows are used to construct
similarity matrices
\cite{Aldroubi_CURSubspace2019}. 
CUR has also become a useful computational primitive in large-scale imaging
and data-recovery problems. Robust CUR methods accelerate robust principal
component analysis and have been applied to video foreground--background
separation and face modeling
\cite{Cai_RapidRobustPCA2021,Cai_RobustCUR2021}. Related CUR-based methods
have been developed for matrix completion under cross-concentrated sampling
\cite{Cai_CrossConcentratedSampling2023}, hyperspectral band selection
\cite{Henneberger_HyperspectralCUR2023}, and graph-regularized imaging
problems \cite{Huang_DualGraphCUR2023}. The same principle extends to
multidimensional data: mode-wise tensor CUR decompositions provide
interpretable low-multilinear-rank representations
\cite{Cai_ModeWiseTensorCUR2021}, while robust tensor CUR and t-CUR methods
have been used for tensor robust PCA and coseparable nonnegative tensor
factorization
\cite{Cai_RobustTensorCUR2024,Chen_CoseparableTensorCUR2024}. 
In scientific computing, cross approximation is particularly useful when a
matrix is available through an entry-evaluation routine rather than as a
stored dense array. Adaptive cross approximation has consequently become a
standard component of data-sparse matrix methods for boundary-integral
operators, collocation discretizations, and hierarchical matrices
\cite{Bebendorf_ApproximationBoundaryElement2000,
Bebendorf_AdaptiveLowRankApproximation2003,
Bebendorf_AcceleratingGalerkinBEM2006}. In these applications, only a small
number of adaptively selected rows and columns need to be evaluated.

\subsection{Adaptive cross approximation}

Adaptive cross approximation constructs the index sets $I$ and $J$
iteratively. Suppose that $I_k$ and $J_k$ contain the pivots selected after
$k$ iterations, and define the residual
\[
    E_k
    =
    A-A(:,J_k)A(I_k,J_k)^{-1}A(I_k,:).
\]
The standard greedy strategy chooses the next pivot by \[  (\widehat{i}_k,\widehat{j}_k)  \in \operatorname*{arg\,max}_{i,j}|(E_k)_{ij}|.\]
The residual is then  updated via \(E_{k+1}  = E_k - E_k(:,\widehat{j}_k) (E_k)_{\widehat{i}_k,\widehat{j}_k}^{-1}E_k(\widehat{i}_k,:).\)
This is equivalent to Gaussian elimination with complete pivoting or an incomplete LU factorization. It accesses only selected rows and columns and avoids explicitly forming the inverse of the full pivot block.

Greedy pivoting may also be interpreted as an incremental attempt to increase
the volume of the selected intersection. If $A(I,J)$ is a maximum-volume
$k\times k$ submatrix, then the associated cross approximation satisfies
\[
    \left\|
    A-A(:,J)A(I,J)^{-1}A(I,:)
    \right\|_{\max}
    \leq
    (k+1)\sigma_{k+1}
\]
\cite{GoreinovMaximalVolumeConcept}. Finding a maximum-volume submatrix is
not only NP-hard, but does not admit any polynomial time approximations  \cite{Civril_SelectingMaximumVolume2009}, which motivates greedy,
local-volume, randomized, and structure-aware pivoting strategies. Some popular examples
include rook pivoting \cite{Foster_GrowthFactorEfficiency1997}, randomly
pivoted Cholesky for positive-semidefinite matrices
\cite{Chen_RandomlyPivotedCholesky2025}, and randomly pivoted LU for general
matrices \cite{Gilles_LowRankApproximationRandomly2026}.

\subsection{Related work}

\paragraph{CUR approximation guarantees.}
The foundational theory of pseudoskeleton and maximum-volume approximation
connects the error of a cross approximation to the volume of its intersection
block and establishes the existence of near-optimal row and column subsets
\cite{Goreinov_TheoryPseudoskeletonApproximations1997,
Zamarashkin_ExistenceNearlyOptimal2018}. Notably, Zamarashkin and Osinsky showed, via a probability argument that there exists row and column selections $I,J$ with error \[
\|A-A(:,J)A(I,J)^{-1}A(I,:)\|_F^2 \leq (1+k)\sum_{\ell=k+1}^n \sigma_{\ell}^2.
\]
These results have been extended to
oversampled and rectangular intersections, average-case analysis, and sharper
bounds involving more of the singular spectrum
\cite{Osinsky_Zamarashkin_Pseudoskeleton2018,
Mikhalev_RectangularMaxvol2018,
Zamarashkin_Osinsky_MaxvolAverage2021,
DeHoog_Hegland_Pseudoskeleton2023,
Osinsky_CloseOptimalCross2025}.

A complementary literature develops constructive CUR algorithms. Randomized
sampling methods provide relative Frobenius-error guarantees
\cite{Drineas_RelativeErrorCUR2008,Wang_AdaptiveCURNystrom2013,
Boutsidis_OptimalCUR2017}, while deterministic and hybrid approaches use
approximate singular subspaces, rank-revealing factorizations, sketches, or
spectrum-revealing pivoting
\cite{Sorensen_DEIMCUR2016,Cortinovis_FrobeniusSubset2020,
Voronin_EfficientCUR2017,Dong_RandomizedCUR2023,
Anderson_SpectralGapCUR2015,Chen_SpectrumRevealingCUR2020}.
Recent work based on generalized interlacing families gives additional
spectral-norm guarantees for CUR and row subset selection
\cite{Cai_GeneralizedInterlacingCUR2025}. These methods provide strong global
error bounds, but they often use oversampling, global subspace information, or
a core matrix different from the inverse of a square pivot block. Error bounds for such methods typically take the form \[
\|A - A(:,J)ZA(I,:)\|_\xi^2 \leq (1+\epsilon)\|A-A^k\|_\xi^2,
\]
where $A^k$ is the truncated SVD, and $\|\cdot\|_\xi$ is some unitary invariant norm (e.g., the Frobenius or operator norm). The number of rows/columns actually sampled is some polynomial in $k$ and $1/\epsilon$. 

\paragraph{Exactness, perturbation, and stability.}
Rank-based characterizations identify when selected rows and columns give an
exact CUR decomposition \cite{Hamm_PerspectivesCUR2020}. Sampling stability,
perturbation bounds, and oversampling analyses describe how CUR behaves under
noise and ill-conditioning
\cite{Hamm_StabilitySamplingCUR2020,Hamm_PerturbationsCUR2021,
Park_CURStabilityOversampling2025}. A local perturbation analysis further
relates fixed-index, rank-truncated CUR to a sampling-induced oblique
tangent-space projector \cite{Huang_LocalTangentCUR2026}. These results analyze
the quality of prescribed index sets; in this work we will mostly focus on the geometric effect of adding an additional ACA pivot.

\paragraph{ACA and pivot selection.}
Classical ACA convergence analyses arise in the approximation of non-local operators and  use regularity assumptions on
the underlying kernel
\cite{Bebendorf_ApproximationBoundaryElement2000,
Bebendorf_AdaptiveLowRankApproximation2003}; kernel-independent results
tend to be much weaker and often overestimate the error
\cite{Bauer_KernelIndependentACA2022}. For general matrices, complete-pivoting
bounds are often tied to maximum volume and Gaussian-elimination growth factors
\cite{Cortinovis_MaximumVolumeSubmatrices2019}. In particular, it has been shown that the greedy algorithm achieves error \[
\|A - A(:,J)A(I,J)^{-1}A(I,:)\|_{\text{max}} \leq 4^k \rho_k \sigma_{k+1},
\]
where $\rho_k$ is the growth factor of Gaussian elimination with complete pivoting, shown by Wilkinson to be less than $2\sqrt{k+1}(k+1)^{\log(k+1)/4}$ \cite{Wilkinson_ErrorAnalysisDirect1961}. 
Related analyses for
positive-semidefinite matrices include pivoted Cholesky and deterministic
maximum-volume selection
\cite{Harbrecht_LowrankApproximationPivoted2012,Massei_SPSDMaxvol2022}. Randomly pivoted Cholesky and LU replace the largest-entry rule with residual-dependent sampling, where pivots are selected with probability proportional to the diagonal magnitude (RPC) or entry magnitude (RPLU). These strategies admit expected-error guarantees a factor of $2^k$ away from optimal \cite{Chen_RandomlyPivotedCholesky2025,Gilles_LowRankApproximationRandomly2026}. Adaptive randomized pivoting provides
a related framework for column selection and cross approximation based on leverage score sampling
\cite{Cortinovis_AdaptiveRandomizedPivoting2026}. 
Existing analyses therefore control CUR and ACA through singular values,
submatrix volume, conditioning, growth factors, or expected residual norms. Without imposing a strict structure on the matrix (such as double diagonal dominance) or oversampling, error analyses always admit an exponential factor of at least $2^k$. This error growth is believed to be asymptotically tight for general matrices \cite{Harbrecht_LowrankApproximationPivoted2012}. 

Existing error analysis does not directly explain the local geometric mechanism by which some pivots
can strongly effect many other rows and columns, nearly annihilating them. The exterior-algebraic
framework developed here is intended to make this mechanism explicit.

\subsection{Contributions and organization}

We develop a geometric framework for ACA by analyzing $A$ as an inner-product matrix. The SVD provides us with a guaranteed decomposition from which we can analyze the entries of $A$. As we will see in section \ref{sec:galerkin_experiments}, any decomposition (including a non-orthogonal decomposition) is sufficient. Let
\(    A=U\Sigma V^\top,
\)
and let $u_i$ and $v_j$ denote the rows of $U$ and $V$, respectively. Then \(  A_{ij}=\langle u_i,v_j\rangle_\sigma,\) where $\langle\cdot,\cdot\rangle_\sigma$ is the inner product weighted by the singular values. Note that $u_i$ and $v_j$ are not the left and right singular vectors, but rows where the $i$th element comes from the $i$th singular vector. Lifting this identity to the exterior powers gives a
representation of each CUR residual entry as an inner product between two
$(k+1)$-blades, divided by the determinant of a normalized pivot block.

The main contributions are as follows.
\begin{enumerate}
    \item
    We derive an exterior-algebraic formula for the entrywise CUR residual.
    The formula separates the weighted magnitudes of the unselected
    singular-vector rows, the interaction between the selected and
    unselected blades, and the normalized volume of the pivot block.

    \item
    We use this representation to recover singular-value error estimates and
    to characterize exact and approximate annihilation of rows and columns
    beyond the selected pivot indices. The analysis shows that the relevant
    phenomenon is linear dependence, or near dependence, among leading
    singular-vector coordinates. We also specialize the framework to symmetric
    positive-semidefinite matrices for a more explicit bound.

    \item
    For a rank-one update, we reduce the exterior-algebraic expression to a
    bound involving weighted sines and cosines. This provides a clear geometric understanding of the residual, allowing us to better understand the effectiveness of ACA in terms of the singular value decay \textit{and} and singular vector geometry. 

    \item
    We apply the geometric interpretation to an asymmetric spiral-kernel
    example in which standard greedy pivoting performs poorly. For radial
    basis Galerkin stiffness matrices, the same interpretation motivates a
    weighted-mass pivoting strategy that seeks to remove a large amount of
    correlated residual mass rather than merely selecting the largest
    residual entry.
\end{enumerate}

Section \ref{sec:all_theory} develops the exterior-algebraic residual
framework, and Sections \ref{sec:understand_k_blades}--\ref{sec:single_bound}
study its geometric and rank-one consequences. Section
\ref{sec:weighted_mass_pivot} introduces the weighted-mass pivoting rule, and
Section \ref{sec:numerics} presents the numerical experiments. Appendix
\ref{apx:cosine_vs_coherence} records additional observations concerning
weighted cosine and singular-vector coherence, while Appendix
\ref{apx:nearest-neighbor} describes the nearest-neighbor preprocessing used
in the implementation.

\section{Exterior-algebraic framework}\label{sec:all_theory}

\subsection{Notation and preliminaries}\label{sec:main_res_and_prelim}

Let $A\in\R^{n\times m}$ have rank $r$ and thin SVD \( A=U\Sigma V^\top, U\in\R^{n\times r},V\in\R^{m\times r},\Sigma=\operatorname{diag}(\sigma_1,\ldots,\sigma_r), \)
with $\sigma_1\geq\cdots\geq\sigma_r>0$. We write $u_i\in\R^r$ for the $i$th row of $U$ and $v_j\in\R^r$ for the $j$th row of $V$. Define the $\sigma$-weighted inner product and norm by
\[
\inner{x,y}_\sigma=x^\top\Sigma y
=\sum_{q=1}^r\sigma_qx_qy_q,
\qquad
\|x\|_\sigma=\sqrt{\inner{x,x}_\sigma}.
\]
Whenever $x$ and $y$ have nonzero weighted norm, their weighted angle is determined by
\[
\cos(\angle_\sigma(x,y))
=\frac{\inner{x,y}_\sigma}{\|x\|_\sigma\|y\|_\sigma}.
\]
The SVD then gives the cross-Gram representation
\(
A_{ij}=\inner{u_i,v_j}_\sigma.
\) 
For a positive-semidefinite matrix, $U=V$ may be chosen and this becomes an ordinary Gram representation.

For $q\geq1$, the exterior power $\bigwedge^q\R^r$ is generated by simple $q$-vectors, or $q$-blades, denoted by \(x_1\wedge\cdots\wedge x_q.\) 
The bilinear product $x\wedge y$ is called the exterior or wedge product of $x$ and $y$. The weighted inner product on simple blades is defined by
\[
\inner{x_1\wedge\cdots\wedge x_q,
       y_1\wedge\cdots\wedge y_q}_\sigma
=
\det\!\big[\inner{x_a,y_b}_\sigma\big]_{a,b=1}^q,
\]
and extended bilinearly. In particular, the weighted norm of a blade is the weighted volume of the parallelotope spanned by its factors. The alternating multilinear vector space $\bigwedge^q \R^r$ has dimension ${n\choose k}$, and the direct sum of all such subspaces for $1\leq q\leq n$ is the exterior algebra of $\R^r$. For a more in-depth discussion of exterior algebra, \cite{Winitzki_LinearAlgebraExterior2009} is a comprehensive reference. 

Exterior algebra provides the geometric language used throughout the remainder of the paper. To streamline the presentation, Table \ref{tab:notation} summarizes the notation that will appear most frequently.  

\begin{table}[t]
\caption{Principal notation.}
\label{tab:notation}
\centering
\small
\renewcommand{\arraystretch}{1.15}
\begin{tabular}{@{}ll@{}}
\toprule
Symbol & Meaning \\
\midrule
$A\in\R^{n\times m}$ & Matrix to be approximated \\
$I=(i_1,\ldots,i_k)$, $J=(j_1,\ldots,j_k)$ & Ordered row and column pivot tuples \\
$S=A(I,J)$ & $k\times k$ pivot block \\
$C=A(:,J)$, $R=A(I,:)$ & Selected columns and rows \\
$u_i$, $v_j$ & Rows of the left and right singular-vector matrices \\
$\inner{x,y}_\sigma$, $\|x\|_\sigma$ & Singular-value-weighted inner product and norm \\
$\angle_\sigma(x,y)$ & Angle induced by $\inner{\cdot,\cdot}_\sigma$ \\
$u_I=u_{i_1}\wedge\cdots\wedge u_{i_k}$ & Blade generated by the selected left singular-vector rows \\
$v_J=v_{j_1}\wedge\cdots\wedge v_{j_k}$ & Blade generated by the selected right singular-vector rows \\
$\bar S$ & Normalized pivot block, $\bar S_{ab}=\cos(\angle_\sigma(u_{i_a},v_{j_b}))$ \\
$\mathbb{I}$ & Identity matrix \\
\bottomrule
\end{tabular}
\end{table}

We now connect these geometric notions with the algebraic structure of skeleton approximation. The key starting point is the classical determinantal identity for the skeleton residual \cite{Cortinovis_MaximumVolumeSubmatrices2019}.  If
\(
E=A-CS^{-1}R,
\)
then, with the new row and column appended to the ordered pivot tuples,
\[
E_{\ell p}
=
\frac{\det A(I+\ell,J+p)}{\det A(I,J)}.
\]
Here and throughout, $I+\ell=(i_1,...,i_k, \ell)$, so the $+$ operator appends the new index to the ordered tuple. 
The numerator and denominator are determinants of weighted cross-Gram matrices, so we can transform the identity into a geometric quantity via exterior algebra. 

The results are developed in the order in which they are used. Theorem
\ref{thm:basic_res_form} gives the exterior-algebraic residual formula and its normalized form. The projected-dependence estimate in Theorem
\ref{thm:k_blade_inner_bound1} then yields the classical max-norm estimate in Corollary
\ref{corr:sig_k_1_bound} and the exact-annihilation criterion in Corollary
\ref{corr:zero_M_sing_val}. Lemmas \ref{lem:simple_denom_det_bound} and
\ref{lem:denom_lower_bound} study the normalized pivot determinant. Theorem
\ref{thm:sym_res_trace} gives the positive-semidefinite trace estimate and  Theorem
\ref{thm:single_iter_cos_bound} specializes the residual geometry to one ACA step. We begin with the residual formula.

\subsection{The determinantal residual as a blade inner product}\label{sec:main_useful_lemma}

We now establish the residual formula. For an increasing multi-index
$K=(q_1<\cdots<q_s)$, we write
\[
e_K=e_{q_1}\wedge\cdots\wedge e_{q_s},
\qquad
\sigma_K=\prod_{q\in K}\sigma_q.
\]
If $X=[x_1\ \cdots\ x_s]\in\R^{r\times s}$, then the coordinate expansion of a simple blade is
\begin{equation}\label{eq:wedge-coordinate-expansion}
x_1\wedge\cdots\wedge x_s
=
\sum_{|K|=s}\det X(K,:)\,e_K.
\end{equation}
Consequently,
\begin{equation}\label{eq:weighted-blade-coordinate-inner}
\inner{x_1\wedge\cdots\wedge x_s,
       y_1\wedge\cdots\wedge y_s}_\sigma
=
\sum_{|K|=s}\sigma_K\det X(K,:)\det Y(K,:).
\end{equation}
  Observe that \eqref{eq:weighted-blade-coordinate-inner} is the Cauchy--Binet expansion of the determinant definition of the induced inner product.

\begin{theorem} \label{thm:basic_res_form}
Let $I=(i_1,\ldots,i_k)$ and $J=(j_1,\ldots,j_k)$ be ordered pivot tuples, and assume that
$S=A(I,J)$ is nonsingular. Define
\[
C=A(:,J),\qquad R=A(I,:),\qquad E=A-CS^{-1}R,
\]
and let
\(
u_I=u_{i_1}\wedge\cdots\wedge u_{i_k}, 
v_J=v_{j_1}\wedge\cdots\wedge v_{j_k}.
\)
Then, for every $\ell\in[n]$ and $p\in[m]$,
\begin{equation}\label{eq:raw-blade-residual}
E_{\ell p}
=
\frac{\inner{u_I\wedge u_\ell,\,v_J\wedge v_p}_\sigma}
     {\det S}.
\end{equation}
Suppose additionally that all weighted norms appearing below are nonzero. Set
\[
\widehat u_i=\frac{u_i}{\|u_i\|_\sigma},
\qquad
\widehat v_j=\frac{v_j}{\|v_j\|_\sigma},
\qquad
\bar S_{ab}=\inner{\widehat u_{i_a},\widehat v_{j_b}}_\sigma.
\]
Then
\begin{equation}\label{eq:normalized-blade-residual}
E_{\ell p}
=
\frac{\|u_\ell\|_\sigma\|v_p\|_\sigma
\inner{\widehat u_I\wedge\widehat u_\ell,
       \widehat v_J\wedge\widehat v_p}_\sigma}
{\det\bar S},
\end{equation}
where $\widehat u_I=\widehat u_{i_1}\wedge\cdots\wedge\widehat u_{i_k}$ and similarly for $\widehat v_J$.
\end{theorem}

\begin{proof}
Appending row $\ell$ and column $p$ to the pivot block gives
\[
A(I+\ell,J+p)
=
\begin{pmatrix}
S & A(I,p)\\
A(\ell,J) & A_{\ell p}
\end{pmatrix}.
\]
The Schur-complement identity therefore yields
\begin{equation}\label{eq:determinantal-residual}
\det A(I+\ell,J+p)
=
\det(S)\big(A_{\ell p}-A(\ell,J)S^{-1}A(I,p)\big)
=
\det(S)E_{\ell p}.
\end{equation}
Because $A_{ij}=\inner{u_i,v_j}_\sigma$, the augmented submatrix is a weighted cross-Gram matrix. By the definition of the induced exterior inner product,
\[
\det A(I+\ell,J+p)
=
\inner{u_I\wedge u_\ell,\,v_J\wedge v_p}_\sigma.
\]
Combining this identity with \eqref{eq:determinantal-residual} proves \eqref{eq:raw-blade-residual}.

For the normalized formula, factor one weighted norm from every row and column of $S$:
\[
\det S
=
\left(\prod_{a=1}^k\|u_{i_a}\|_\sigma\right)
\left(\prod_{b=1}^k\|v_{j_b}\|_\sigma\right)
\det\bar S.
\]
Multilinearity of the wedge product gives the same selected-row and selected-column factors in the numerator of \eqref{eq:raw-blade-residual}, together with the additional factors $\|u_\ell\|_\sigma$ and $\|v_p\|_\sigma$. Cancelling the common factors proves \eqref{eq:normalized-blade-residual}.
\end{proof}

Observe also that the denominator of \eqref{eq:normalized-blade-residual} can be written $\inner{\widehat u_I, \widehat v_J}_\sigma$, providing a slightly more consistent formula. There will not be any qualitative difference between analyzing the determinantal version and the $k$-blade version of the denominator, so the determinantal version will typically be used. 

Theorem \ref{thm:basic_res_form} exposes a distinction that is hidden in the unnormalized determinant. The magnitudes of the selected singular-vector rows cancel between the numerator and denominator; what remains is the normalized pivot volume $\det\bar S$. The unselected row and column magnitudes remain, while the interaction between selected and unselected indices is encoded by the two $(k+1)$-blades. Thus, maximizing $|\det S|$ and maximizing $|\det\bar S|$ need not favor the same pivots. We note that this result does not automatically imply that maximizing $|\det\bar{S}|$ is a better strategy than maximizing $|\det S|$. As we will see, selecting pivots of extremely small $\sigma$-norm can cause blow-up in the numerator. 

We next record two elementary consequences of the coordinate expansion. They will be used repeatedly below.

\begin{lemma} \label{lem:exterior_lin_dep}
Let $x_1,\ldots,x_s\in\R^r$, and let $P_M:\R^r\to\R^M$ be the projection onto the first $M$ coordinates. If
$P_Mx_1,\ldots,P_Mx_s$ are linearly dependent, then
\[
\det X(K,:)=0
\qquad
\text{for every }K\subseteq[M]\text{ with }|K|=s,
\]
where $X=[x_1\ \cdots\ x_s]$. Equivalently, every coordinate of
$x_1\wedge\cdots\wedge x_s$ supported entirely in the first $M$ directions vanishes.
\end{lemma}

\begin{proof}
For $K\subseteq[M]$, the matrix $X(K,:)$ is a row submatrix of
$[P_Mx_1\ \cdots\ P_Mx_s]$. The latter has rank less than $s$, so every $s\times s$ minor is zero. The conclusion follows from \eqref{eq:wedge-coordinate-expansion}.
\end{proof}

\begin{lemma} \label{lem:k_blade_inner_form}
For $X=[x_1\ \cdots\ x_s]$ and $Y=[y_1\ \cdots\ y_s]$,
\[
\inner{x_1\wedge\cdots\wedge x_s,
       y_1\wedge\cdots\wedge y_s}_\sigma
=
\sum_{|K|=s}\sigma_K\det X(K,:)\det Y(K,:).
\]
In particular,
\(
\|x_1\wedge\cdots\wedge x_s\|_\sigma^2
=
\sum_{|K|=s}\sigma_K\det X(K,:)^2.
\)
\end{lemma}

\begin{proof}
The first identity is \eqref{eq:weighted-blade-coordinate-inner}; setting $X=Y$ gives the second.
\end{proof}

These formulas provide insight into the precise effect of the singular spectrum, as a blade coordinate indexed by $K$ is weighted by the product $\sigma_K$. If leading-coordinate dependence forces all coordinates with $K\subseteq[M]$ to vanish, the non-zero terms must all contain some tail singular values. 

\subsection{Geometry of the residual blades}\label{sec:understand_k_blades}

The numerator in \eqref{eq:normalized-blade-residual} depends on two related features: the magnitudes of the residual blades and the angle between them. We first record an idealized calculation that helps interpret the angular term, and then prove a deterministic estimate that links blade magnitude to the tail of the singular spectrum.

\subsubsection{An idealized random-blade model}\label{sec:k_blade_ang}

The space $\bigwedge^s\R^r$ has dimension $\binom{r}{s}$. A simple blade generated by $s$ vectors occupies a highly structured subset of this space, so its coordinates are not independent. Nevertheless, an independent-coordinate model provides a useful baseline for the scale of a weighted blade inner product.

\begin{theorem} \label{thm:k_blade_avg_inner}
Let \( X=\sum_{|K|=s}X_K e_K,  Y=\sum_{|K|=s}Y_K e_K,\) where all $X_K$ and $Y_K$ are independent standard Gaussian variables. Then
\[\Exp\inner{X,Y}_\sigma^2= \sum_{|K|=s}\sigma_K^2,
\qquad \Exp\|X\|_\sigma^2 = \Exp\|Y\|_\sigma^2=\sum_{|K|=s}\sigma_K.
\]
\end{theorem}

\begin{proof}
In the weighted exterior basis,
\(
\inner{X,Y}_\sigma=\sum_{|K|=s}\sigma_KX_KY_K.
\)
After squaring and taking expectations, all cross terms vanish by independence and centering, while
$\Exp(X_K^2Y_K^2)=1$. This gives the first identity. The second follows directly from
$\|X\|_\sigma^2=\sum_K\sigma_KX_K^2$ and $\Exp X_K^2=1$.
\end{proof}

When the numerator and the two norms concentrate around their typical scales, Theorem \ref{thm:k_blade_avg_inner} suggests the heuristic
\begin{equation}\label{eq:random-blade-angle-heuristic}
\cos^2\angle_\sigma(X,Y)
\approx
\frac{\sum_{|K|=s}\sigma_K^2}
     {\left(\sum_{|K|=s}\sigma_K\right)^2}.
\end{equation}
For uniform weights, the right-hand side is $1/\binom{r}{s}$, the familiar inverse-dimension scale. Equation \eqref{eq:random-blade-angle-heuristic} is an simple model rather than a theorem about the simple blades generated by singular-vector rows; those blades have far fewer degrees of freedom and strongly dependent coordinates. Numerically, it was found that equation \eqref{eq:random-blade-angle-heuristic} is an accurate estimate for the expected weighted cosines of $k$-blades generated by random gaussian vectors, when the singular values do not decay too quickly. This is to be expected from the domain in which the CLT applies.

\subsubsection{Leading-coordinate dependence and tail weights}

Dependence among the first $M$ singular coordinates removes every blade component supported entirely in those coordinates. For a bivector this can be seen directly. Let
\(
x=(1,2,2,3,4,5,2)^\top, 
y=(2,4,4,1,-1,0,2)^\top.
\)
Their first three coordinates are proportional. Under the matrix representation
$x\wedge y=xy^\top-yx^\top$, one obtains
\[
\begin{pmatrix}
0&0&0&-5&-9&-10&-2\\
0&0&0&-10&-18&-20&-4\\
0&0&0&-10&-18&-20&-4\\
5&10&10&0&-7&-5&4\\
9&18&18&7&0&5&10\\
10&20&20&5&-5&0&10\\
2&4&4&-4&-10&-10&0
\end{pmatrix}.
\]
The leading $3\times3$ block vanishes. Lemma \ref{lem:exterior_lin_dep} is the higher-order version of this observation.

We now quantify its effect on the weighted residual blades.

\begin{theorem}\label{thm:k_blade_inner_bound1}
Assume $\sigma_1=1$. Let $I=(i_1,\ldots,i_k)$ and $J=(j_1,\ldots,j_k)$, and suppose the Euclidean norms of the vectors
$u_\ell,u_{i_1},\ldots,u_{i_k},v_p,v_{j_1},\ldots,v_{j_k}$ are at most one. Let $M_1,M_2\geq k$ satisfy
\(P_{M_1}u_\ell\in\Span\{P_{M_1}u_{i_1},\ldots,P_{M_1}u_{i_k}\}\) and \(P_{M_2}v_p\in\Span\{P_{M_2}v_{j_1},\ldots,P_{M_2}v_{j_k}\}.\)
Then
\begin{equation}\label{eq:projected-dependence-bound}
\left|
\inner{u_I\wedge u_\ell,\,v_J\wedge v_p}_\sigma
\right|
\leq
\sqrt{\sigma_{M_1+1}\sigma_{M_2+1}}.
\end{equation}
\end{theorem}

\begin{proof}
By Cauchy--Schwarz in $\bigwedge^{k+1}\R^r$,
\[
\left|\inner{u_I\wedge u_\ell,v_J\wedge v_p}_\sigma\right|
\leq
\|u_I\wedge u_\ell\|_\sigma
\|v_J\wedge v_p\|_\sigma.
\]
We bound the first factor; the second is identical. Let
$X=[u_{i_1}\ \cdots\ u_{i_k}\ u_\ell]$. By Lemma \ref{lem:exterior_lin_dep},
$\det X(K,:)=0$ whenever $K\subseteq[M_1]$ and $|K|=k+1$. Hence Lemma \ref{lem:k_blade_inner_form} gives
\[
\|u_I\wedge u_\ell\|_\sigma^2
=
\sum_{\substack{|K|=k+1,K\nsubseteq[M_1]}}
\sigma_K\det X(K,:)^2.
\]
Every index set in this sum contains a coordinate larger than $M_1$. Since $\sigma_1=1$ and the singular values are nonincreasing,
$\sigma_K\leq\sigma_{M_1+1}$. Therefore
\[
\|u_I\wedge u_\ell\|_\sigma^2
\leq
\sigma_{M_1+1}
\sum_{|K|=k+1}\det X(K,:)^2
=
\sigma_{M_1+1}\|u_I\wedge u_\ell\|_2^2.
\]
The unweighted blade norm is the volume of the parallelotope generated by the columns of $X$. Hadamard's inequality and the assumed vector-norm bounds imply
$\|u_I\wedge u_\ell\|_2\leq1$. Thus
$\|u_I\wedge u_\ell\|_\sigma^2\leq\sigma_{M_1+1}$.
\end{proof}

Rows of the thin singular-vector matrices automatically satisfy the Euclidean norm hypothesis because $U$ and $V$ have orthonormal columns. The proof also shows the unnormalized version
\[
\left|
\inner{u_I\wedge u_\ell,\,v_J\wedge v_p}_\sigma
\right|
\leq
\sigma_1^k\sqrt{\sigma_{M_1+1}\sigma_{M_2+1}}.
\]
The bound deliberately discards the angle between the two residual blades. It is therefore conservative, but it identifies the spectral scale that remains after leading-coordinate dependence removes the dominant blade components. In practice, the angle between $k$-blades will tend to be small when the singular vectors decay rapidly (see section \ref{sec:k_blade_ang}), so this approximation is often not a significant sacrifice. 

Taking $M_1=M_2=k$ yields a familiar consequence. Any $k+1$ vectors in $\R^k$ are linearly dependent, so the hypotheses are automatic.

\begin{corollary} \label{corr:sig_k_1_bound}
Let $S=A(I,J)$ be a nonsingular $k\times k$ pivot block. If $\sigma_1=1$, then
\[
\big\|A-A(:,J)S^{-1}A(I,:)\big\|_{\max}
\leq
\frac{\sigma_{k+1}}{|\det S|}.
\]
\end{corollary}

\begin{proof}
Apply Theorem \ref{thm:k_blade_inner_bound1} with $M_1=M_2=k$ to the numerator in \eqref{eq:raw-blade-residual}, bounding every element of the residual, and thus the max norm.
\end{proof}

Note that the unnormalized bound will have a factor of $\sigma_1^k$ in the numerator, but normalizing $\sigma_1\to 1$ is equivalent to diving the denominator determinant by $\sigma_1^k$. Thus, the CUR residual will be the same as in the normalized case. 

This estimate can also be obtained from compound-matrix or singular-value interlacing arguments. The exterior formulation adds the more refined information that $M_1$ and $M_2$ may be larger than $k$. When the leading coordinates remain dependent in a higher-dimensional subspace, the relevant tail singular values can be much smaller than $\sigma_{k+1}$.

The limiting case gives exact annihilation.

\begin{corollary} \label{corr:zero_M_sing_val}
Let $S=A(I,J)$ be nonsingular. Suppose $M\geq k$, $\sigma_{M+1}=0$, and
\[
P_Mu_\ell\in
\Span\{P_Mu_{i_1},\ldots,P_Mu_{i_k}\}.
\]
Then \(
\big(A-A(:,J)S^{-1}A(I,:)\big)_{\ell p}=0\) for every \(p\). 
The analogous condition on $v_p$ annihilates the entire $p$th residual column.
\end{corollary}

\begin{proof}
Because $\sigma_{M+1}=0$, all nonzero exterior weights are supported in the first $M$ coordinates. Lemma \ref{lem:exterior_lin_dep} makes every such coordinate of $u_I\wedge u_\ell$ vanish, so the numerator of \eqref{eq:raw-blade-residual} is zero.
\end{proof}

\paragraph{Example.}
As an example of the preceding corollary, consider the matrix
\(
    A=U\Sigma V^\top,
\)
where $U=V=Q$,
\[
Q=
\frac{1}{2}
\begin{pmatrix}
1 & 1 & 1 & 1\\
1 & -1 & 1 & -1\\
1 & 1 & -1 & -1\\
1 & -1 & -1 & 1
\end{pmatrix},
\qquad
\Sigma=
\begin{pmatrix}
1 & 0 & 0 & 0\\
0 & \frac12 & 0 & 0\\
0 & 0 & 0 & 0\\
0 & 0 & 0 & 0
\end{pmatrix}.
\]
Since $Q$ is orthogonal and $\Sigma$ has nonnegative diagonal
entries in nonincreasing order, this is a singular value
decomposition of $A$. Direct calculation gives
\[
A=
\frac{1}{8}\begin{pmatrix}
3 & 1 & 3 & 1\\
1 & 3 & 1 & 3\\
3 & 1 & 3 & 1\\
1 & 3 & 1 & 3
\end{pmatrix}.
\]
Select the pivot $(i,j)=(1,3)$. Since $A_{1,3}=3/8$, the
corresponding rank-one residual is
\[
A-A(:,3)A_{1,3}^{-1}A(1,:)
=
\begin{pmatrix}
0 & 0 & 0 & 0\\
0 & \frac13 & 0 & \frac13\\
0 & 0 & 0 & 0\\
0 & \frac13 & 0 & \frac13
\end{pmatrix}.
\]
Thus, in addition to the selected first row and third column, the
third row and first column are also annihilated.

This behavior follows directly from
Corollary~\ref{corr:zero_M_sing_val}. Let $P_2$ denote projection
onto the first two coordinates. If $u_i$ and $v_j$ denote the rows
of $U$ and $V$, respectively, then \(P_2u_3=P_2u_1,  P_2v_1=P_2v_3. \)
Hence,
\[
    P_2u_3\in\operatorname{span}\{P_2u_1\},
    \qquad
    P_2v_1\in\operatorname{span}\{P_2v_3\}.
\]
Because $\sigma_3=0$, selecting the pivot $(1,3)$ therefore also
annihilates the third row and first column.

\subsection{The normalized pivot determinant}\label{sec:denom_det}

Next, we analyze the denominator of \eqref{eq:normalized-blade-residual}. The matrix
\[
\bar S_{ab}=\cos\angle_\sigma(u_{i_a},v_{j_b})
\]
is a cross-Gram matrix of weighted-unit vectors. Its determinant measures the mutual alignment of the two selected $k$-frames under the geometry imposed by the $\sigma$-weighted inner product.

A universal upper bound follows immediately from the exterior interpretation.

\begin{lemma} \label{lem:simple_denom_det_bound}
For every normalized pivot block $\bar S$,
\(
|\det\bar S|\leq1.
\)
\end{lemma}

\begin{proof}
Let $\widehat u_I$ and $\widehat v_J$ denote the blades generated by the weighted-unit selected rows. Then
\(
\det\bar S=\inner{\widehat u_I,\widehat v_J}_\sigma.
\)
Cauchy--Schwarz gives
\[
|\det\bar S|
\leq
\|\widehat u_I\|_\sigma\|\widehat v_J\|_\sigma.
\]
Each blade norm is the volume spanned by unit vectors and is at most the product of their norms, hence at most one.
\end{proof}

The upper bound is attained when the two selected weighted frames are orthonormal (under the $\sigma$ inner product) and coincide up to an orthogonal change of orientation. Basis-aligned coherent singular-vector rows provide one transparent example, but equality is not unique to that case.

The next result gives a concrete, if deliberately strong, sufficient condition for the normalized determinant to be bounded away from zero. It formalizes the situation in which each selected row and column is concentrated on a distinct singular coordinate.

\begin{lemma}\label{lem:denom_lower_bound}
Let $\widehat u_1,\ldots,\widehat u_k$ and
$\widehat v_1,\ldots,\widehat v_k$ be vectors in $\R^r$ satisfying
\(
    \|\widehat u_a\|_\sigma
    =
    \|\widehat v_a\|_\sigma
    =
    1,
     a=1,\ldots,k,
\)
and define the normalized cross-Gram matrix
\(
    \bar S_{ab}
    =
    \inner{\widehat u_a,\widehat v_b}_\sigma
    =
    \sum_{q=1}^r
    \sigma_q(\widehat u_a)_q(\widehat v_b)_q.
\)
Let $\mu,\nu\in[0,1]$, and suppose that there are distinct coordinates
$q_1,\ldots,q_k\in[r]$ such that, for every $a$,
\[
    \sigma_{q_a}(\widehat u_a)_{q_a}^2
    \geq \mu,
    \qquad
    \sigma_{q_a}(\widehat v_a)_{q_a}^2
    \geq \nu.
\]
Suppose also that, for every $q\neq q_a$,
\(\sigma_q(\widehat u_a)_q^2  \leq  \frac{1-\mu}{r-1},    \sigma_q(\widehat v_a)_q^2 \leq\frac{1-\nu}{r-1}.\)
Define
\[
    d(\mu,\nu)
    =
    \sqrt{\mu\nu}
    -
    \sqrt{(1-\mu)(1-\nu)}
\text{~and~}
    b(\mu,\nu,r)
    =
    \sqrt{\frac{1-\nu}{r-1}}
    +
    \sqrt{\frac{1-\mu}{r-1}} +
    \frac{r-2}{r-1}
    \sqrt{(1-\mu)(1-\nu)}.
\]
Finally, let
\(
    \delta
    =
    d(\mu,\nu)
    -
    (k-1)b(\mu,\nu,r).
\)
If $\delta>0$, then \(|\det\bar S|\geq\delta^k.\)
\end{lemma}
\begin{proof}
Set $x_a=\Sigma^{1/2}\widehat u_a$ and
$y_a=\Sigma^{1/2}\widehat v_a$. Then
$\|x_a\|_2=\|y_a\|_2=1$ and
$\bar S_{ab}=x_a^\top y_b$.

For a diagonal entry, the reverse triangle inequality and
Cauchy--Schwarz give
\[
\begin{aligned}
|\bar S_{aa}|
&\geq
|x_{a,q_a}y_{a,q_a}|
-
\left|
\sum_{q\neq q_a}x_{a,q}y_{a,q}
\right| \\
&\geq
\sqrt{\mu\nu}
-
\left(\sum_{q\neq q_a}x_{a,q}^2\right)^{1/2}
\left(\sum_{q\neq q_a}y_{a,q}^2\right)^{1/2} \geq
\sqrt{\mu\nu}
-
\sqrt{(1-\mu)(1-\nu)}
=
d(\mu,\nu).
\end{aligned}
\]

Now let $a\neq b$. Since $q_a\neq q_b$, the coordinate $q_a$ is
nondistinguished for $y_b$, the coordinate $q_b$ is nondistinguished
for $x_a$, and every other coordinate is nondistinguished for both.
Using $\|x_a\|_2=\|y_b\|_2=1$ and the coordinatewise assumptions,
\[
\begin{aligned}
|\bar S_{ab}|
&\leq
|x_{a,q_a}y_{b,q_a}|
+
|x_{a,q_b}y_{b,q_b}|
+
\sum_{q\notin\{q_a,q_b\}}|x_{a,q}y_{b,q}| \\
&\leq
\sqrt{\frac{1-\nu}{r-1}}
+
\sqrt{\frac{1-\mu}{r-1}}
+
\frac{r-2}{r-1}\sqrt{(1-\mu)(1-\nu)} =
b(\mu,\nu,r).
\end{aligned}
\]

Since $\delta>0$, we have $d(\mu,\nu)>0$, so every diagonal entry of
$\bar S$ is nonzero. Let
\[
D=
\operatorname{diag}
\bigl(
\operatorname{sgn}(\bar S_{11}),\ldots,
\operatorname{sgn}(\bar S_{kk})
\bigr),
\qquad
B=D\bar S.
\]
Then $|\det B|=|\det\bar S|$, while
$B_{aa}=|\bar S_{aa}|\geq d(\mu,\nu)$ and
\(
\sum_{b\neq a}|B_{ab}|
\leq
(k-1)b(\mu,\nu,r).
\)
Hence, by the Gershgorin disk theorem, every eigenvalue $\lambda$ of
$B$ satisfies
\[
\lambda
\geq
d(\mu,\nu)-(k-1)b(\mu,\nu,r)
=
\delta,
\]
and therefore $|\lambda|\geq\delta$. It follows that
\[
|\det\bar S|
=
|\det B|
=
\prod_{a=1}^k|\lambda_a|
\geq
\delta^k.
\]
\end{proof}

The concentration assumptions in Lemma \ref{lem:denom_lower_bound} are
restrictive, so the result is best interpreted as a geometric certificate for
the normalized determinant. Figure \ref{fig:det_bound} shows the symmetric
specialization $\mu=\nu$ for $r=1000$. As $k$ increases, a stronger
concentration is required for the bound to become positive because each
Gershgorin radius contains $k-1$ off-diagonal contributions. Once this
threshold is exceeded, the bound increases rapidly toward one.

\begin{figure}[ht]
    \centering
\includegraphics[width=0.55\linewidth]{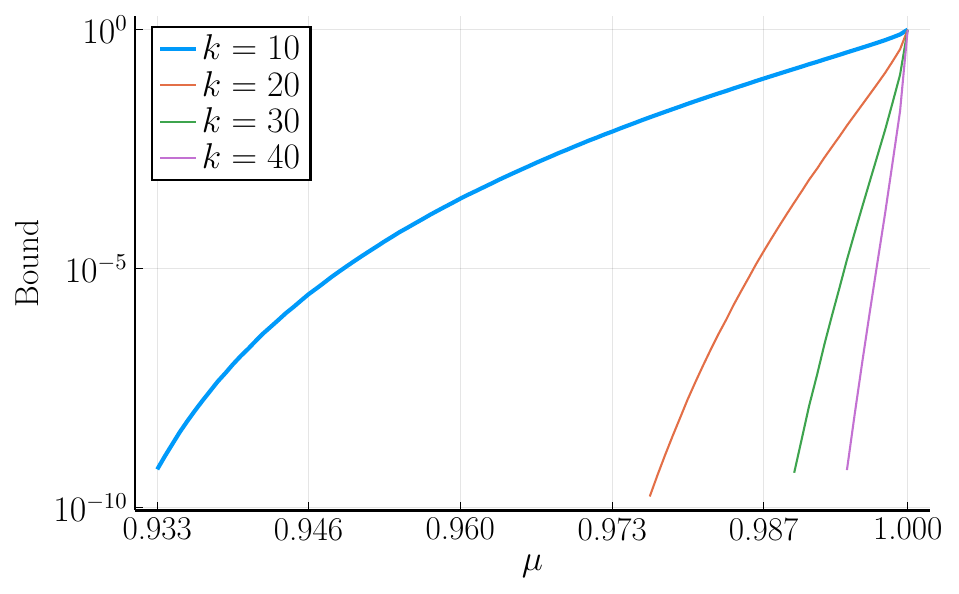}
    \caption{The determinant lower bound $\delta^k$ from Lemma \ref{lem:denom_lower_bound} along the symmetric slice  $\mu=\nu$, for $r=1000$ and several values of $k$.}
    \label{fig:det_bound}
\end{figure}

Normalization does not eliminate the effect of the weighted row magnitudes.
Equation \eqref{eq:normalized-blade-residual} still contains
$\|u_\ell\|_\sigma$ and $\|v_p\|_\sigma$, while a small weighted norm of a
selected row can amplify its normalized tail coordinates. A practical pivot
rule must therefore balance angular geometry, row magnitudes, and the
conditioning of the normalized pivot block.

A complementary question is which singular-vector geometries minimize the
largest attainable normalized determinant. The rank-one analysis in
Appendix \ref{apx:cosine_vs_coherence} suggests highly incoherent
configurations as natural candidates; the corresponding problem for higher
exterior powers remains open, but we believe the answer will similarly be highly incoherent singular vectors. 

\subsection{Positive-semidefinite matrices}\label{sec:psd}

Let
\(
A=U\Sigma U^\top,
\)
where $U\in\R^{n\times n}$ is orthogonal and $\Sigma$ is diagonal and nonnegative. If $u_i$ denotes the $i$th row of $U$, then
$A_{ij}=\inner{u_i,u_j}_\sigma$. For a diagonal pivot $(i,i)$, the singular vector rows are guaranteed to be well aligned. The normalized cosine we will see in Theorem \ref{thm:single_iter_cos_bound} is one. Moreover,
\[
A-A(:,i)A_{ii}^{-1}A(i,:)
\]
is positive semidefinite whenever $A_{ii}>0$.

For positive-semidefinite matrices, diagonal pivoting is natural and is optimal for several determinant-based criteria  \cite{Cortinovis_MaximumVolumeSubmatrices2019}. This does not imply that a diagonal rank-one cross update minimizes every residual norm over all possible row--column pairs. The following small example illustrates the distinction. Let
\[
H=
\begin{pmatrix}
1&1&-1&-1\\
1&-1&-1&1\\
1&1&1&1\\
1&-1&1&-1
\end{pmatrix},
\qquad
\Lambda=\operatorname{diag}(1,\epsilon,\epsilon^2,\epsilon^4),
\qquad
A=H^\top\Lambda H.
\]
For every $\epsilon>0$, $A$ is positive definite. By the symmetry of the Hadamard construction, all four diagonal pivots give the same Frobenius residual. At $\epsilon=2.9$,
\[
\left\|A-A(:,i)A_{ii}^{-1}A(i,:)\right\|_F
\approx68.393
\qquad\text{for every }i,
\]
whereas the off-diagonal pivot $(1,3)$ gives
\(
\left\|A-A(:,3)A_{13}^{-1}A(1,:)\right\|_F
\approx60.100.
\)
Thus diagonal pivoting can remain preferable for structure preservation and stability without being the unconstrained Frobenius-optimal rank-one cross update.

We next give a sufficient condition under which a specified diagonal pivot is no worse than an off-diagonal competitor.

\begin{theorem} \label{thm:diag_pivoting_spd}
Let $A=U\Sigma U^\top$, where $U$ is orthogonal and
$0\preceq\Sigma\preceq \mathbb{I}$. Fix indices $i,j$ such that
$A_{ii}> 0$ and $A_{ij}\neq0$, and write
\[x=u_i^\top\Sigma u_j,\qquad b=u_i^\top\Sigma^3u_j, \qquad d_i=u_i^\top\Sigma^2u_i, \qquad d_j=u_j^\top\Sigma^2u_j.\]
Suppose
\( |b|\leq\epsilon_3, d_id_j=\epsilon_2^2,\)
and
\begin{equation}\label{eq:diag-pivot-sufficient-interval}
|x|\leq\sqrt{\epsilon_3^2+\epsilon_2^2}-\epsilon_3.
\end{equation}
Then
\(\left\|A-A(:,i)A(i,i)^{-1}A(i,:)\right\|_F \leq\left\|A-A(:,j)A(i,j)^{-1}A(i,:)\right\|_F.
\)
\end{theorem}

\begin{proof}
For any nonzero pivot $A_{ij}=x$, direct expansion of the Frobenius norm gives
\begin{equation}\label{eq:psd-offdiag-residual-closed-form}
\left\|A-A(:,j)x^{-1}A(i,:)\right\|_F^2
=
\|A\|_F^2-\frac{2b}{x}+\frac{d_id_j}{x^2}.
\end{equation}
Indeed,
$\inner{A,A(:,j)A(i,:)}_F=(A^3)_{ij}=b$, while
$\|A(:,j)A(i,:)\|_F^2=(A^2)_{jj}(A^2)_{ii}=d_jd_i$.

Set
\(
a=u_i^\top\Sigma u_i, e=u_i^\top\Sigma^3u_i.
\)
For the diagonal pivot,
\[
\left\|A-A(:,i)a^{-1}A(i,:)\right\|_F^2
=
\|A\|_F^2-\frac{2e}{a}+\frac{d_i^2}{a^2}.
\]
Both $e$ and $a$ are positive as $\Sigma$ has positive entries, and hence
\begin{equation}\label{eq:diag-residual-upper}
\left\|A-A(:,i)a^{-1}A(i,:)\right\|_F^2
\leq
\|A\|_F^2+\frac{d_i^2}{a^2}.
\end{equation}
Because $\Sigma^2\preceq\Sigma$, we have $d_i/a\leq1$. Comparing
\eqref{eq:diag-residual-upper} with \eqref{eq:psd-offdiag-residual-closed-form}, it is therefore sufficient that \(x^2+2bx\leq d_id_j. \)
The hypotheses imply \(x^2+2bx\leq x^2+2\epsilon_3|x|\leq\epsilon_2^2=d_id_j,\) 
where the second inequality is equivalent to \eqref{eq:diag-pivot-sufficient-interval}.
\end{proof}

When $\epsilon_3\ll\epsilon_2$, condition \eqref{eq:diag-pivot-sufficient-interval} is approximately
$|u_i^\top\Sigma u_j|^2\lesssim d_id_j$. It resembles Cauchy--Schwarz, but the comparison involves $\Sigma^2$ in the right-hand-side magnitudes and therefore depends on the spectral decay.

For diagonal pivots, trace error is especially useful because the residual remains positive semidefinite. The next estimate relates this error to the leading eigenvector coordinates. The bound below yields a squared conditioning number. 

\begin{theorem} \label{thm:sym_res_trace}
Let $A=U\Sigma U^\top$, where $U\in\R^{n\times n}$ is orthogonal and
$\sigma_1\geq\cdots\geq\sigma_n\geq0$. Let
$I=(i_1,\ldots,i_k)$ be diagonal pivot indices, assume $A(I,I)$ is nonsingular, and let $P_k$ be the projection onto the first $k$ coordinates. Define
\[
X=
\begin{pmatrix}
|&&|\\
u_{i_1}^\top&\cdots&u_{i_k}^\top\\
|&&|
\end{pmatrix}\in\R^{n\times k},
\qquad
U^{(0)}=P_kX\in\R^{k\times k}.
\]
If $U^{(0)}$ is nonsingular, then
\begin{equation}\label{eq:psd-trace-bound}
\operatorname{tr}\!\left(A-A(:,I)A(I,I)^{-1}A(I,:)\right)
\leq
(n-k)\left(1+\|(U^{(0)})^{-1}\|_2\right)^2
\sum_{q=k+1}^n\sigma_q.
\end{equation}
\end{theorem}

 \begin{proof}
     Let $E=A-A(:,I)A(I,I)^{-1}A(I,:)$. We will bound $E_{\ell,\ell}$, for all $\ell$ and then obtain a bound on the trace. Note that the $P_k u_{i_p}$ make up a basis for $\R^k$. So there exists $\alpha\in\R^k$ such that 
     \[
     P_k u_{\ell} = \alpha_1 P_k u_{i_1} + ... + \alpha_k P_k u_{i_k}.
     \]
      Thus, $P_k u_\ell = U^{(0)} \alpha$ and $\|\alpha\|_2 \leq \|(U^{(0)})^{-1}\|_2 \|P_k u_\ell\|_2 \leq \|(U^{(0)})^{-1}\|_2$. Now consider the vector $z_\ell = u_\ell - X\alpha$ whose first $k$ coordinates vanish. From Theorem~\ref{thm:basic_res_form}, $E_{\ell,\ell}$ can be represented as 
      \[
     E_{\ell,\ell} = \frac{\|u_\ell\wedge u_I\|_\sigma^2}{\|u_I\|_\sigma^2} = \frac{\|z_{\ell}\wedge u_I\|_\sigma^2}{\|u_I\|_\sigma^2}\leq \|z_\ell\|_\sigma^2.
     \]
     The second equality comes from the fact that the exterior product is unchanged by adding multiples of $u_{i_p}$ to $u_\ell$ for any $p$, and the last inequality comes from the fact that the norm of the $k+1$-blade is the volume of the parallelopiped spanned by the vectors, which is always less than the product of the norms of said vectors. 

     Finally, observe that $z_\ell$ vanishes in the first $k$ coordinates, and the max-norm of the remaining coordinates can be bounded by \[
     \|z_\ell\|_{\text{max}} \leq \|z_\ell\|_2 \leq \|u_\ell\|_2 + \|X\|_2\|\alpha\|_2 \leq 1 + \|(U^{(0)})^{-1}\|_2.
     \]
     Note that we have used the fact that the $u_{i_p}$'s are orthonormal. We then arrive as \[
     \|z_\ell\|_\sigma^2 \leq (1 + \|(U^{(0)})^{-1}\|_2)^2\sum_{q=k+1}^n\sigma_q.
     \]
     The result follows from observing that there are $n-k$ non-zero diagonal elements, all satisfying the same bound. 
 \end{proof}

The tail sum in \eqref{eq:psd-trace-bound} is the optimal rank-$k$ trace error. The result therefore identifies a geometric quasi-optimality factor: the approximation is favorable when the selected eigenvector rows form a well-conditioned basis for the leading $k$ coordinates. In the best such cases $\|(U^{(0)})^{-1}\|_2\approx 1$, leading to an error which is a factor of $2(n-k)$ from optimal. An interesting feature of this estimate is the fact that it tends to improve as $k$ increases, in contrast to estimates for the error in the greedy algorithm, which grow exponentially.

\subsection{The rank-one update}\label{sec:single_bound}

For a single pivot, the residual blades are bivectors, whose weighted norms
are determined by weighted angles. This gives a particularly transparent
form of the residual estimate. 
Let
\[
E^{(\widehat i,\widehat j)}
=
A-A(:,\widehat j)A_{\widehat i,\widehat j}^{-1}A(\widehat i,:)
\]
denote the residual after selecting the nonzero pivot
$(\widehat i,\widehat j)$.

\begin{theorem}\label{thm:single_iter_cos_bound}
Suppose $A_{\widehat i,\widehat j}\neq0$. Then, for every $\ell$ and $p$
for which the weighted angles below are defined,
\begin{equation}\label{eq:rank-one-angle-bound}
\left|E^{(\widehat i,\widehat j)}_{\ell p}\right|
\leq
\frac{
\|u_\ell\|_\sigma\|v_p\|_\sigma
\left|\sin\angle_\sigma(u_\ell,u_{\widehat i})\sin\angle_\sigma(v_p,v_{\widehat j})\right|
}{
\left|\cos\angle_\sigma(u_{\widehat i},v_{\widehat j})\right|
}.
\end{equation}
\end{theorem}

\begin{proof}
Apply Theorem \ref{thm:basic_res_form} with $k=1$. After normalizing the
four singular-vector rows, the denominator is
$\cos\angle_\sigma(u_{\widehat i},v_{\widehat j})$. For weighted-unit
vectors $x$ and $y$,
\[
\|x\wedge y\|_\sigma^2
=
\det
\begin{pmatrix}
1 & \inner{x,y}_\sigma\\
\inner{x,y}_\sigma & 1
\end{pmatrix}
=
1-\inner{x,y}_\sigma^2
=
\sin^2\angle_\sigma(x,y).
\]
Cauchy--Schwarz applied to the two bivectors in the numerator of
\eqref{eq:normalized-blade-residual} now gives
\eqref{eq:rank-one-angle-bound}.
\end{proof}

Theorem \ref{thm:single_iter_cos_bound} provides a concrete way to view the approximate annihilation of residual rows/columns. Row $\ell$ is annihilated when $u_\ell$ is nearly parallel under the $\sigma$-inner product to $u_{\widehat i}$.

\begin{corollary} 
\label{corr:cur_res_from_cosines}
Let $\varepsilon_u,\varepsilon_v\in[0,1]$, and suppose
\[
\cos\angle_\sigma(u_\ell,u_{\widehat i})
\geq1-\varepsilon_u,
\qquad
\cos\angle_\sigma(v_p,v_{\widehat j})
\geq1-\varepsilon_v.
\]
Then
\begin{equation}\label{eq:nearly-parallel-rank-one-bound}
\left|E^{(\widehat i,\widehat j)}_{\ell p}\right|
\leq
\frac{
\|u_\ell\|_\sigma\|v_p\|_\sigma
\sqrt{
(2\varepsilon_u-\varepsilon_u^2)
(2\varepsilon_v-\varepsilon_v^2)
}
}{
\left|\cos\angle_\sigma(u_{\widehat i},v_{\widehat j})\right|
}.
\end{equation}
In particular, if $\varepsilon_u=\varepsilon_v=\varepsilon$, then
\[
\left|E^{(\widehat i,\widehat j)}_{\ell p}\right|
\leq
\frac{
(2\varepsilon-\varepsilon^2)
\|u_\ell\|_\sigma\|v_p\|_\sigma
}{
\left|\cos\angle_\sigma(u_{\widehat i},v_{\widehat j})\right|
}
\leq
\frac{
2\varepsilon\|u_\ell\|_\sigma\|v_p\|_\sigma
}{
\left|\cos\angle_\sigma(u_{\widehat i},v_{\widehat j})\right|
}.
\]
\end{corollary}

\begin{proof}
The assumptions imply
\[
\sin^2\angle_\sigma(u_\ell,u_{\widehat i})
\leq
1-(1-\varepsilon_u)^2
=
2\varepsilon_u-\varepsilon_u^2,
\]
and analogously for $v_p$ and $v_{\widehat j}$. Substitution into
\eqref{eq:rank-one-angle-bound} proves the result.
\end{proof}
Thus, exact parallelism on either side forces the corresponding sine factor
to vanish. If $u_\ell$ is exactly parallel to $u_{\widehat i}$, the entire $\ell$th
residual row is zero (a special case of corrolary \ref{corr:zero_M_sing_val}). The entrywise estimate also yields a Frobenius-norm criterion for comparing
rank-one pivots. Define
\[
L(\widehat i)
=
\sum_{\ell=1}^n
\|u_\ell\|_\sigma^2
\sin^2\angle_\sigma(u_\ell,u_{\widehat i}),
\qquad
R(\widehat j)
=
\sum_{p=1}^m
\|v_p\|_\sigma^2
\sin^2\angle_\sigma(v_p,v_{\widehat j}).
\]
Squaring \eqref{eq:rank-one-angle-bound} and summing over $\ell$ and $p$
gives
\begin{equation}\label{eq:rank-one-frobenius-factorization}
\left\|E^{(\widehat i,\widehat j)}\right\|_F^2
\leq
\frac{
L(\widehat i)R(\widehat j)
}{
\cos^2\angle_\sigma(u_{\widehat i},v_{\widehat j})
}.
\end{equation}

The quantities $L(\widehat i)$ and $R(\widehat j)$ measure the weighted
left and right singular-vector mass not aligned with the selected pivot
rows. A useful pivot therefore has two properties: its left and right rows
represent a large amount of the remaining singular-vector mass, and the two
pivot rows are fairly aligned. Consequently, maximizing
the pivot magnitude, or even maximizing the weighted cosine alone, need not
minimize the residual.

For the independent-coordinate model with $s=1$, Theorem
\ref{thm:k_blade_avg_inner} suggests
\[
\cos^2\angle_\sigma(x,y)
\approx
\frac{\sum_q\sigma_q^2}{\left(\sum_q\sigma_q\right)^2}
=
\frac{1}{d_{\mathrm{eff}}},
\qquad
d_{\mathrm{eff}}
=
\frac{\left(\sum_q\sigma_q\right)^2}{\sum_q\sigma_q^2}.
\]
Thus, rapid singular-value decay reduces the effective weighted dimension,
even when the ambient dimension is large. For a positive-semidefinite matrix
with a diagonal pivot, the denominator in
\eqref{eq:rank-one-frobenius-factorization} is equal to one. In that setting,
the pivot problem reduces to selecting a row that aligns with a large amount
of the remaining weighted mass. This observation serves as the basis for the 
weighted-mass pivoting rule introduced in Section
\ref{sec:weighted_mass_pivot}.

\section{Weighted-mass pivoting}\label{sec:weighted_mass_pivot}

The rank-one analysis in Section \ref{sec:single_bound} suggests that a useful
pivot should not be judged solely by its magnitude. A pivot may produce a
large reduction in the residual when its associated row and column are aligned
with many other residual rows and columns. Computing the weighted
singular-vector angles directly would require an SVD, so we instead construct
a local proxy for this collective-alignment effect.

Suppose that $A\in\R^{n\times n}$ is symmetric positive semidefinite matrix. This setup is motivated by Galerkin FEM matrices in which a decomposition for $A$ is given by the basis functions and each basis function has an associated center, such that nearby centers lead to high correlation under the $A$ inner product. Let the center for basis element $i$ be $x_i$. Let $B$ denote the
current residual, and let $\mathcal N_i$ be the set of the $\ell$ nearest centers
to $x_i$, with a fixed convention concerning whether $i$ itself is included.
Greedy diagonal pivoting selects the index maximizing $B_{ii}$. In contrast,
weighted-mass pivoting assigns candidate $i$ the score
\begin{equation}\label{eq:weighted-mass-score}
    m_i
    =
    \sum_{j\in\mathcal N_i}
    \bigl(B_{jj}A_{ji}\bigr)^2.
\end{equation}
Here $B_{jj}$ measures the residual mass associated with index $j$, while
$A_{ji}$ serves as a proxy for the correlation between indices $j$ and $i$.
Thus, a candidate receives a large score when it is strongly correlated with
several nearby indices that still carry substantial residual mass.

\begin{algorithm}[t]
\caption{Weighted-mass diagonal ACA}\label{alg:weight_mass_piv}
\begin{algorithmic}[1]
\Require Symmetric positive-semidefinite matrix $A\in\R^{n\times n}$ or an
entry oracle for $A$, centers $X=\{x_i\}_{i=1}^n$, target rank $k_{\max}$,
neighborhood size $\ell$, and tolerance $\tau$
\Ensure Pivot set $I$
\State Precompute $\mathcal N_i\gets\Call{Nearest}{X,i,\ell}$ for
$i=1,\ldots,n$
\State $B\gets A$, $I\gets\varnothing$
\For{$k=1,\ldots,k_{\max}$}
    \For{$i\in[n]\setminus I$}
        \State
        $m_i\gets\displaystyle
        \sum_{j\in\mathcal N_i}\bigl(B_{jj}A_{ji}\bigr)^2$
    \EndFor
    \State
    $p\gets\displaystyle
    \operatorname*{arg\,max}_{i\in[n]\setminus I}m_i$
    \If{$B_{pp}\leq\tau$}
        \State \algorithmicbreak
    \EndIf
    \State $I\gets I\cup\{p\}$
    \State $B\gets B-B(:,p)B_{pp}^{-1}B(p,:)$
\EndFor
\State \Return $I$
\end{algorithmic}
\end{algorithm}

Algorithm \ref{alg:weight_mass_piv} is written using the explicit residual
matrix for clarity. A matrix-free implementation can instead maintain the
residual diagonal and the selected low-rank factors, evaluating the required
entries through an entry oracle. Once the neighborhoods have been
precomputed, evaluating all scores costs $\O(n\ell)$ arithmetic per iteration
and is parallel across the candidate indices. The neighborhood table requires
$\O(n\ell)$ storage, while the diagonal and score vectors require $\O(n)$
storage. The storage required for the ACA factors grows linearly with the
selected rank. 
The nearest-neighbor search is independent of the pivoting rule and may be
implemented using a brute-force search, a spatial tree, or a GPU routine. The
specific insertion-based GPU routine used in the experiments is therefore
reported separately in Appendix \ref{apx:nearest-neighbor}.

The hyperparameter $\ell$ can be viewed as a measurement of the global information considered. With $\ell=1$, the algorithm does a greedy search, but considers $B_{jj}A_{jj}$ rather than the $B_{jj}$ that would come from the typical greedy algorithm. In practice this is shown to preform comparably with the standard greedy algorithm. When $\ell=n$, the algorithm considers the entire diagonal of $E$ and how correlated each element is to the proposed pivot. Taking into account more global information will theoretically improve the approximation, but, as we will see in Section \ref{sec:galerkin_experiments}, often only a few nearby basis elements are needed. 

\section{Numerical experiments}\label{sec:numerics}

We present two numerical studies illustrating the geometric mechanism developed
in Sections \ref{sec:understand_k_blades}--\ref{sec:single_bound}. The first is
a diagnostic example showing why the largest residual entry can be a poor
pivot. The second evaluates the weighted-mass strategy from
Section \ref{sec:weighted_mass_pivot} on positive-semidefinite radial-basis
Galerkin matrices. Throughout this section, approximation quality is measured
by the Frobenius norm of the residual.

\subsection{Asymmetric spiral-kernel diagnostic}
\label{sec:spiral_kernel_analysis}

Consider two point sets
\[
    X=\{x_i\}_{i=1}^{200},
    \qquad
    Y=\{y_j\}_{j=1}^{200},
\]
with points $x_i,y_j\in\R^2$, lying on opposing planar spirals, as shown in
Figure \ref{fig:spiral_points}. We construct the asymmetric inverse-distance
kernel matrix
\begin{equation}\label{eq:spiral-kernel}
    A_{ij}
    =
    k(x_i,y_j),
    \qquad
    k(x,y)
    =
    \frac{\alpha^2}{\|x-y\|_2}.
\end{equation}
Related examples (both symmetric and asymmetric) are known to be difficult for greedy
diagonal pivoting
\cite{Chen_RandomlyPivotedCholesky2025,
Gilles_LowRankApproximationRandomly2026}.

\begin{figure}[ht]
    \centering
    \safeincludegraphics[width=0.60\linewidth]{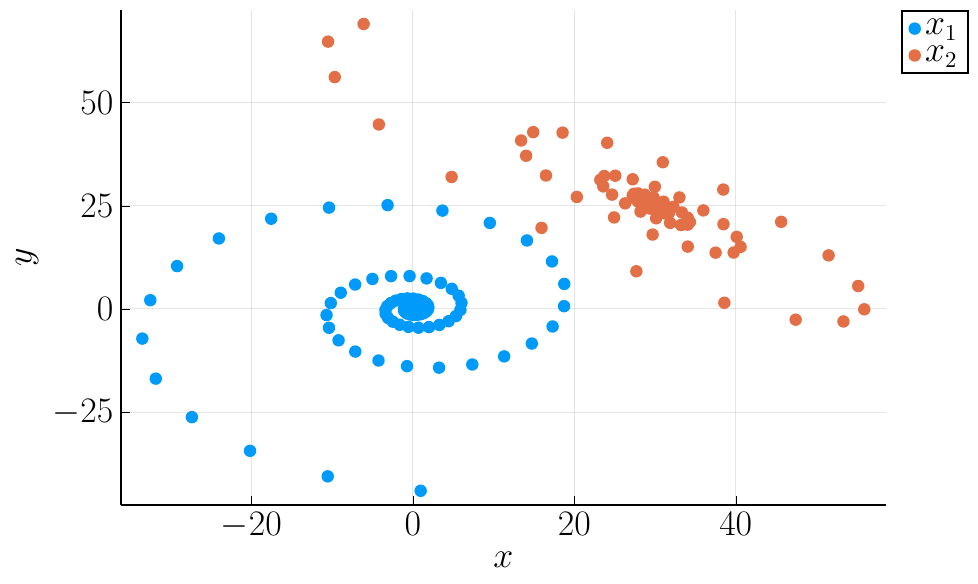}
    \caption{Two spiral point sets used to construct the asymmetric
    inverse-distance kernel matrix.}
    \label{fig:spiral_points}
\end{figure}

Figure \ref{fig:spiral_diagnostics} summarizes the behavior of greedy ACA on
this example. The heat map in Figure \ref{fig:spiral_diagnostics}a reveals two
competing structures. An isolated row and column contain the largest entries,
corresponding to outer points on one spiral that pass close to the opposing
spiral. At the same time, the matrix contains a substantially larger block of
moderately large and nearly constant entries generated by the dense groups of
points near the two spiral centers.

\begin{figure}[t]
    \centering
    \begin{minipage}[t]{0.323\textwidth}
        \centering
        \safeincludegraphics[width=\linewidth]{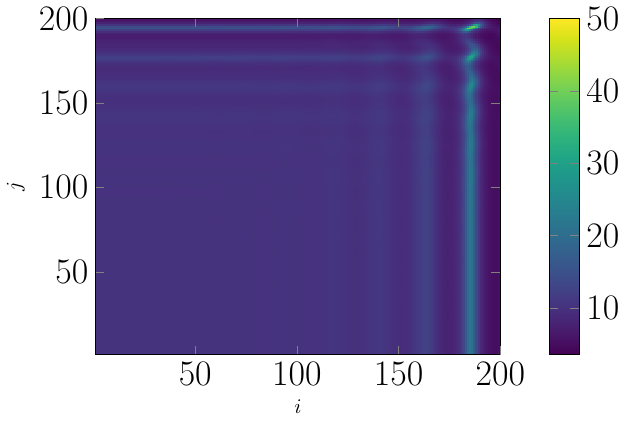}\\[2pt]
        \small (a) Kernel matrix heat map.
        \label{fig:spiral_K_heatmap}
    \end{minipage}\hfill
    \begin{minipage}[t]{0.323\textwidth}
        \centering
        \safeincludegraphics[width=\linewidth]{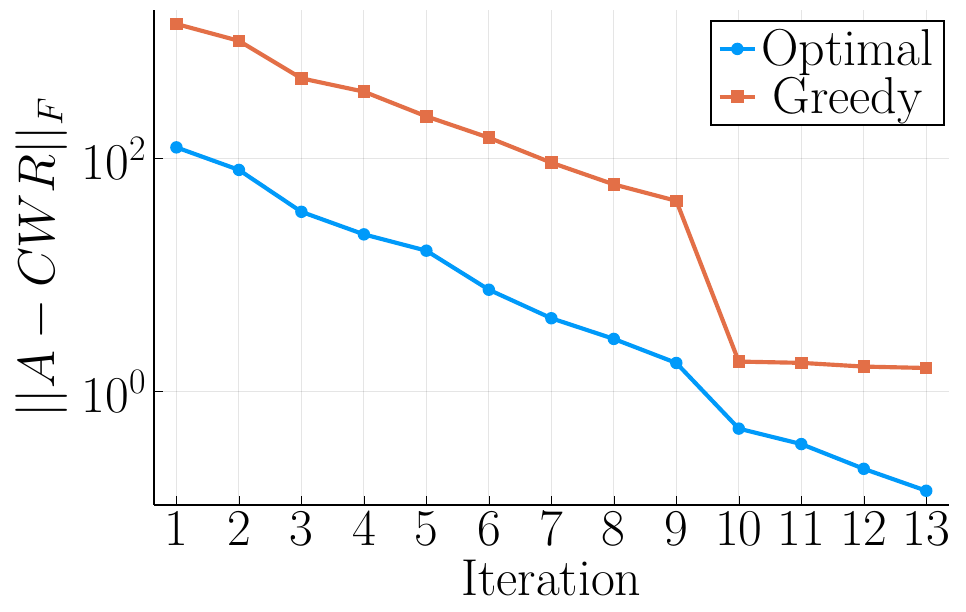}\\[2pt]
        \small (b) Greedy ACA and SVD errors.
        \label{fig:greedy_vs_opt_spiral}
    \end{minipage}\hfill
    \begin{minipage}[t]{0.323\textwidth}
        \centering
        \safeincludegraphics[width=\linewidth]{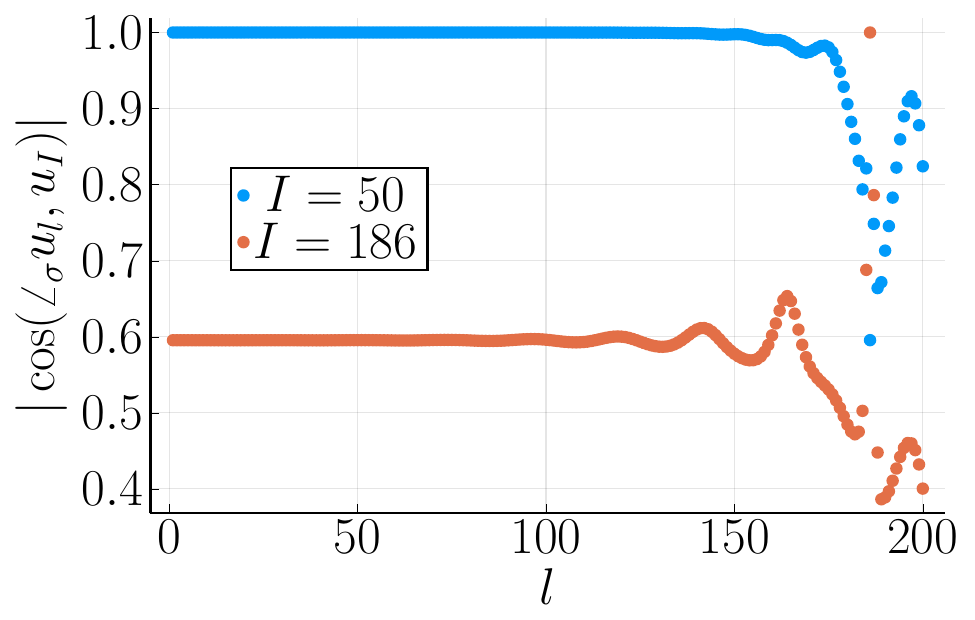}\\[2pt]
        \small (c) Weighted row alignments.
        \label{fig:sig_cos_svs_spiral}
    \end{minipage}

    \caption{Diagnostics for the asymmetric spiral-kernel example.
    The three panels show the matrix structure, the approximation error of
    greedy ACA relative to the truncated-SVD benchmark, and the weighted
    alignment of the candidate singular-vector rows.}
    \label{fig:spiral_diagnostics}
\end{figure}

Although the singular values decay rapidly, Figure
\ref{fig:spiral_diagnostics}b shows that the greedy ACA residual remains well
above the optimal rank-$k$ truncated-SVD error over the displayed range of
ranks. Thus, the poor approximation is not caused by the absence of an
accurate low-rank representation, but by the indices selected by the greedy
pivot rule.

To understand why the greedy pivot provides a poor approximation, consider the first greedy pivot, 
\(
    (\widehat i,\widehat j)=(186,195).
\)
The row $u_{186}$ is poorly aligned with most of the other left
singular-vector rows in the weighted geometry. In contrast, $u_{50}$ is nearly
parallel to a broad collection of rows, while the column index $75$ is chosen
analogously from the correlated column group. Figure
\ref{fig:spiral_diagnostics}c compares the weighted alignments of the
geometry-aware row $u_{50}$ and the greedy row $u_{186}$.

This behavior is consistent with
Corollary \ref{corr:cur_res_from_cosines}: a pivot row and column that are
nearly parallel to many other weighted singular-vector rows can approximately
annihilate a large collection of residual entries. The reported first-step
residuals are
\[
    \|E_{\mathrm{greedy}}\|_F
    \approx 1414.7,
    \qquad
    \|E_{(50,75)}\|_F
    \approx 125.9.
\]
The pivot $(50,75)$ therefore reduces the first-step residual by more than an
order of magnitude. Entrywise magnitude favors the isolated extreme visible
in Figure \ref{fig:spiral_diagnostics}a, whereas the weighted geometry favors a
moderately large pivot that represents many rows and columns simultaneously.

This example is diagnostic rather than a complete algorithmic comparison: it
contrasts two first-step pivots but does not report a full geometry-aware ACA
trajectory over multiple iterations. However, we believe that this example provides an explanation as to why randomized pivot selection such as RPC or RPLU \cite{Chen_RandomlyPivotedCholesky2025, Gilles_LowRankApproximationRandomly2026} can so often out-perform the greedy algorithm on such matrices. In such cases, there are many more pivot selections which are highly correlated, so the diagonal sampling is more likely to pick up one of the correlated pivots ($i\in\{1,...,150\}$ in this case) than the uncorrelated pivot. 

\subsection{Radial-basis Galerkin experiments}
\label{sec:galerkin_experiments}

We next evaluate weighted-mass pivoting on positive-semidefinite stiffness
matrices arising from a radial-basis Galerkin discretization. We mostly follow the classical FEM problem formulation \cite{larson_finite_2010}, but with basis functions and quadrature akin to \cite{shaw_radial_2025}. Consider the
elliptic problem
\[
    -\Delta u+cu=f
    \qquad\text{in }\Omega\subseteq\R^2,
\]
with boundary conditions chosen such that the weak problem is well posed. The
associated bilinear form is
\[
    a(u,v)
    =
    \int_\Omega
    \bigl(\nabla u\cdot\nabla v+cuv\bigr).
\]
A selection of basis functions $\{\varphi_i\}_{i=1}^n$ produces the stiffness matrix
\begin{equation}\label{eq:galerkin-stiffness}
    A_{ij}
    =
    a(\varphi_j,\varphi_i)
    =
    \int_\Omega
    \left(
        \nabla\varphi_i\cdot\nabla\varphi_j
        +
        c\varphi_i\varphi_j
    \right).
\end{equation}

In all numerical experiments below, we use $c(x,y)=\frac{1}{0.1 + |x-y|}$, but similar results can be seen for other choices of $c$. 

Local polynomial finite-element bases generally yield sparse stiffness
matrices, for which a global low-rank approximation may be unnecessary. Here
we instead use globally supported Gaussian radial basis functions
\begin{equation}\label{eq:gaussian-rbf}
    \varphi_i(x)
    =
    \exp\bigl(-\epsilon\|x-x_i\|_2^2\bigr).
\end{equation}
Such bases arise in meshfree radial-basis discretizations of elliptic
problems
\cite{Liu_MeshfreeRadialPoint2005,Kien_RadialBasisFunction2023}.
On nonuniform point sets, nearby basis functions may be strongly correlated
under the energy inner product, making these matrices natural test cases for
the weighted-mass rule.

\subsubsection{Experimental protocol}

The stiffness entries in \eqref{eq:galerkin-stiffness} are evaluated using
seven-point Gaussian quadrature on each triangle, following
\cite{shaw_radial_2025}. We compare the following three pivoting strategies: (1) greedy diagonal pivoting; (2) our weighted-mass pivoting from Algorithm \ref{alg:weight_mass_piv}; (3) randomly pivoted Cholesky (RPC) \cite{Chen_RandomlyPivotedCholesky2025}.

Unless otherwise stated, weighted-mass pivoting uses
neighborhood size $\ell=5$. The neighborhoods are computed once before the ACA
iteration; the preprocessing routine is described in
Appendix \ref{apx:nearest-neighbor}.

\subsubsection{Nonuniform circular point set}
\label{sec:nonuniform_circ}

The first set of experiments uses points of the form
\[(\rho\cos\theta,\rho\sin\theta),\]
where the radial coordinates $\rho$ are logarithmically spaced and the angles $\theta$ are uniformly spaced. This sampling concentrates more points near the outer boundary. Figure~\ref{fig:nonuniform_circ_tri} illustrates a $10\times 10$ example. For the approximation experiments, we use $50$ radial and $50$ angular coordinates, resulting in $2500$ centers. The triangulation is computed using \lstinline{DelaunayTriangulation.jl} \cite{VandenHeuvel2024DelaunayTriangulation}.

\begin{figure}[t]
    \centering
    \safeincludegraphics[width=0.52\linewidth]
    {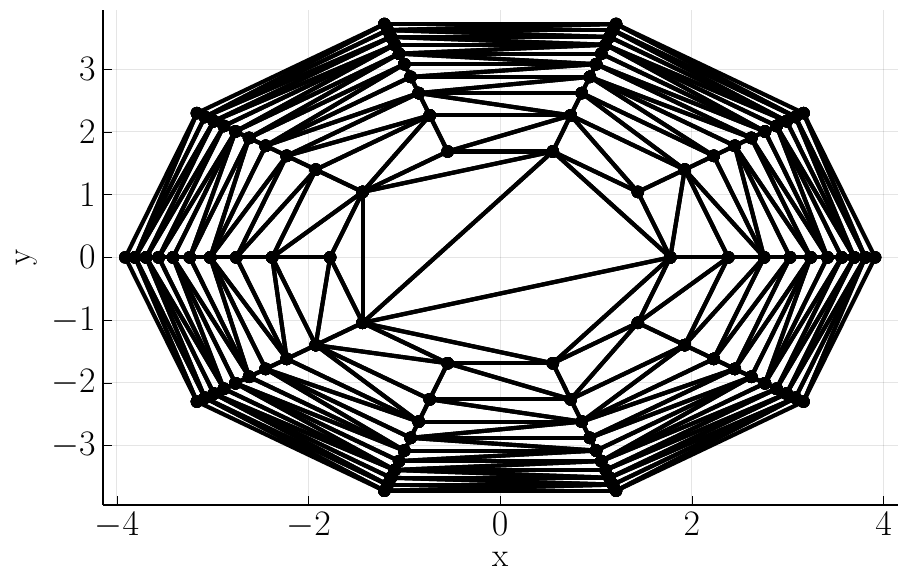}
    \caption{Triangulation of the smaller nonuniform circular point set used
    for visualization.}
    \label{fig:nonuniform_circ_tri}
\end{figure}

\begin{figure}[t]
    \centering
    \begin{minipage}[t]{0.323\textwidth}
        \centering
        \safeincludegraphics[width=\linewidth]
        {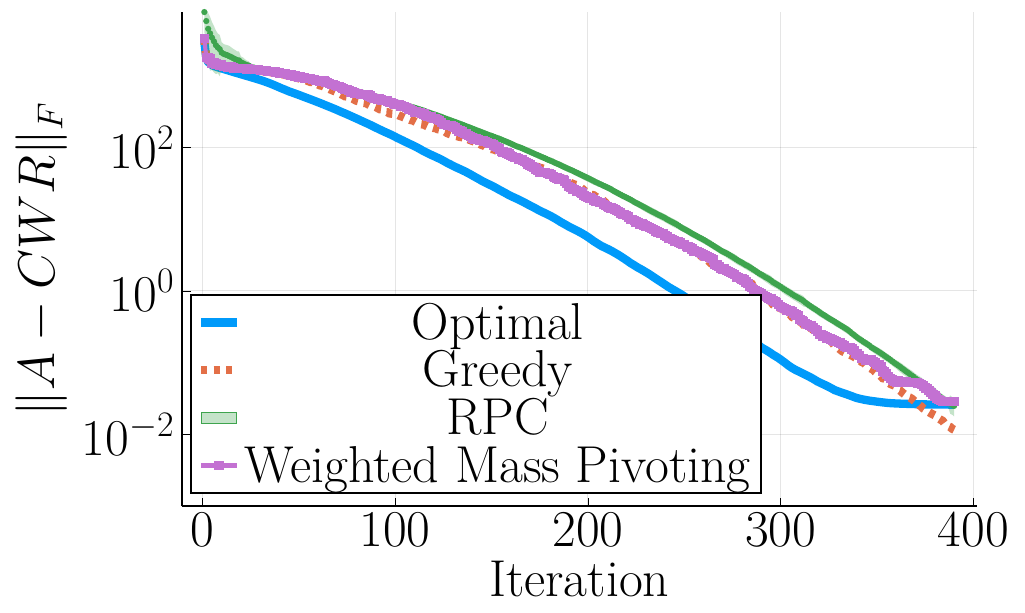}
        \small (a) $\epsilon=3$
        \label{fig:nonuniform-eps3}
    \end{minipage}
    \hfill
    \begin{minipage}[t]{0.325\textwidth}
        \centering
        \safeincludegraphics[width=\linewidth]
        {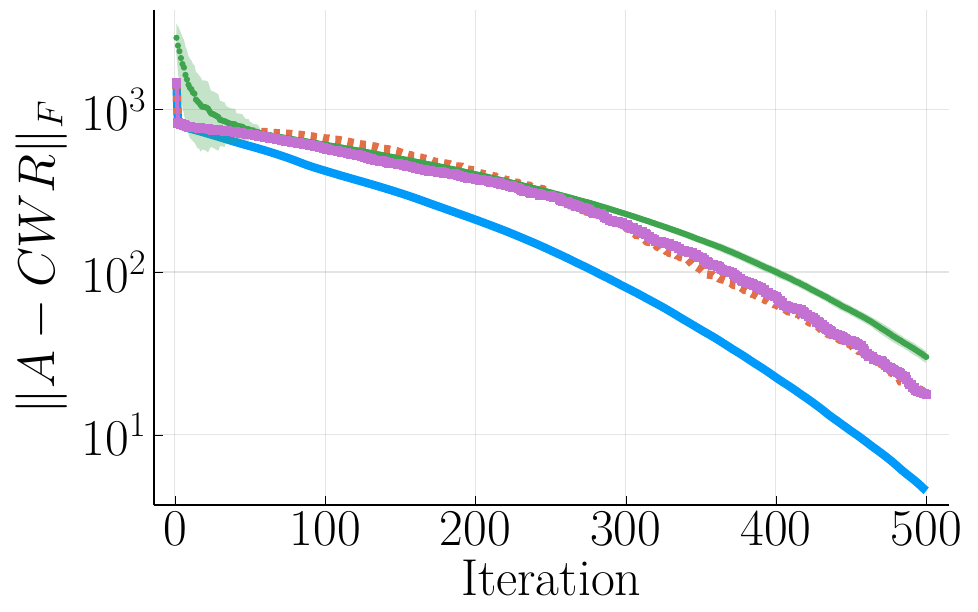}
        \small (b) $\epsilon=11$
        \label{fig:nonuniform-eps11}
    \end{minipage}
    \hfill
    \begin{minipage}[t]{0.325\textwidth}
        \centering
        \safeincludegraphics[width=\linewidth]
        {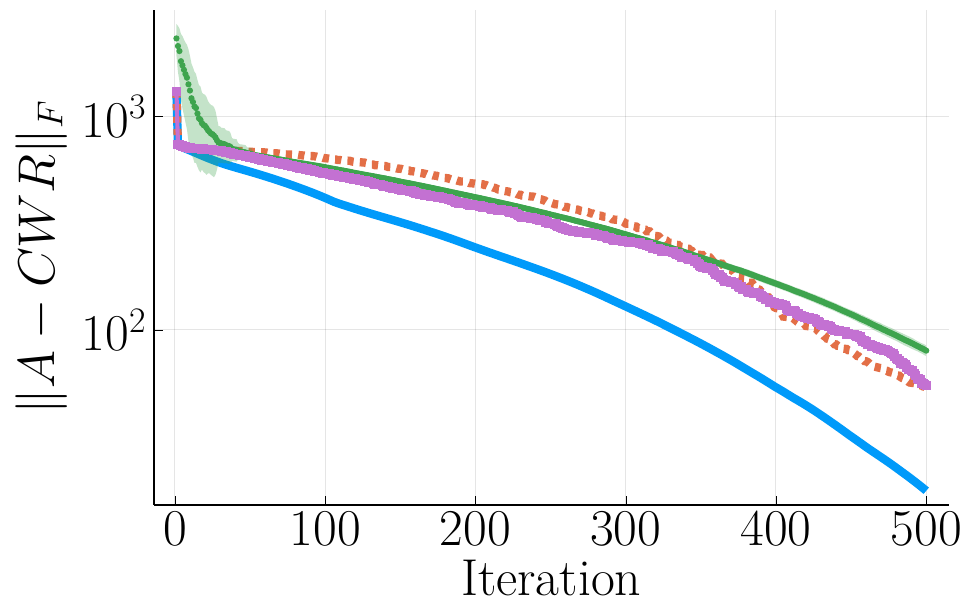}
        \small (c) $\epsilon=15$
        \label{fig:nonuniform-eps15}
    \end{minipage}

    \caption{Frobenius residuals for greedy diagonal pivoting,
    weighted-mass pivoting, and RPC on the $2500$-point nonuniform circular
    domain. Although sometimes hard to see, RPC reports the mean of $100$ runs with one standard deviation
    shaded. Weighted-mass pivoting uses $\ell=5$.}
    \label{fig:non_unif_circ_galerkin_aca}
\end{figure}

We test the radial-basis shape parameters \(    \epsilon=3, \epsilon=11,  \epsilon=15. \) The corresponding Frobenius residual curves are shown in Figure \ref{fig:non_unif_circ_galerkin_aca}.  As shown in Figure \ref{fig:non_unif_circ_galerkin_aca}, the advantage of weighted-mass pivoting is largest for the two larger shape parameters. One possible explanation is that, when $\epsilon$ is small, the broad radial basis functions generate strongly correlated leading modes that are removed rapidly by all three methods. At larger $\epsilon$, the correlations are more localized and persist across a larger number of candidate neighborhoods, giving the weighted-mass score more opportunities to distinguish useful pivots. This interpretation is consistent with the observed curves but is not a direct consequence of the present theory.

In several of the reported tests, RPC and weighted-mass pivoting initially outperform greedy pivoting, after which the greedy method narrows the gap. Similar catch-up behavior has been reported for randomized pivoting \cite{Chen_RandomlyPivotedCholesky2025, Gilles_LowRankApproximationRandomly2026}. A plausible explanation is that, once the most strongly correlated modes have been removed, the remaining off-diagonal correlations become weak and the weighted-mass score approaches an ordinary diagonal-mass criterion. In the reported parameter study, increasing $\ell$ beyond $5$ did not visibly improve the residual curves.

\subsubsection{Clustered point sets}

The second family of experiments samples points from Gaussian clusters,
\[
    x\sim\mathcal N(y_g,s_g^2I),
\]
where $y_g$ and $s_g$ denote the center and standard deviation of cluster
$g$. The stiffness assembly is the same as in
Section \ref{sec:nonuniform_circ}, with radial-basis shape parameter
$\epsilon=10$ and weighted-mass neighborhood size $\ell=5$. We consider three
configurations:

\begin{enumerate}
    \item four clusters containing $500$, $100$, $400$, and $700$ points,
    centered at $(0,0)$, $(2,-1)$, $(-1,-1)$, and $(1,2)$, with standard
    deviations $1$, $0.4$, $0.2$, and $1.4$;

    \item three clusters of $900$ points centered at $(0,0)$, $(2,-2)$,
    and $(-2,-2)$, with standard deviations $0.5$, $0.4$, and $0.4$;

    \item twelve clusters of $150$ points arranged on the grid
    $\{0,2,4,6\}\times\{0,2,4\}$, each with standard deviation $0.35$.
\end{enumerate}

The point configurations and corresponding residual curves are shown in
Figure \ref{fig:point_groups_aca}.

\begin{figure}[t]
    \centering
    \begin{minipage}{0.99\textwidth}
        \begin{minipage}{0.32\textwidth}
        \centering
        \safeincludegraphics[width=\linewidth]{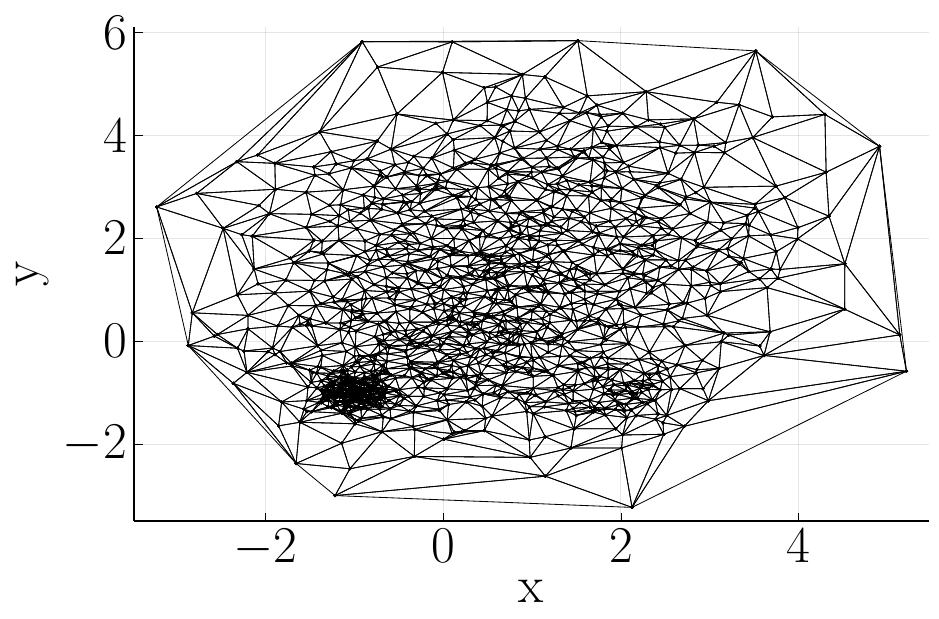}
        \end{minipage}
        \begin{minipage}{0.32\textwidth}
        \centering
        \safeincludegraphics[width=\linewidth]{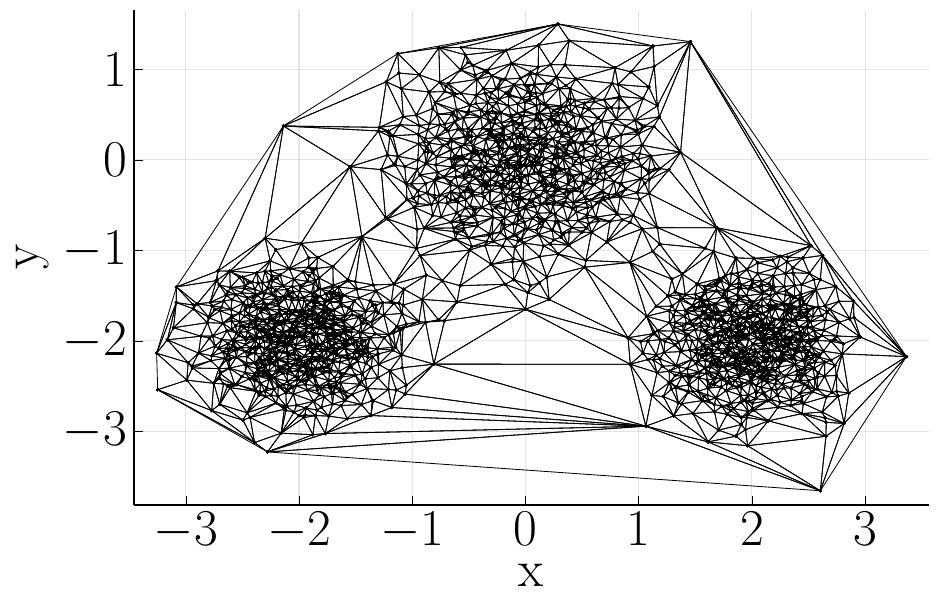}
        \end{minipage}
        \begin{minipage}{0.32\textwidth}
        \centering
        \safeincludegraphics[width=\linewidth]{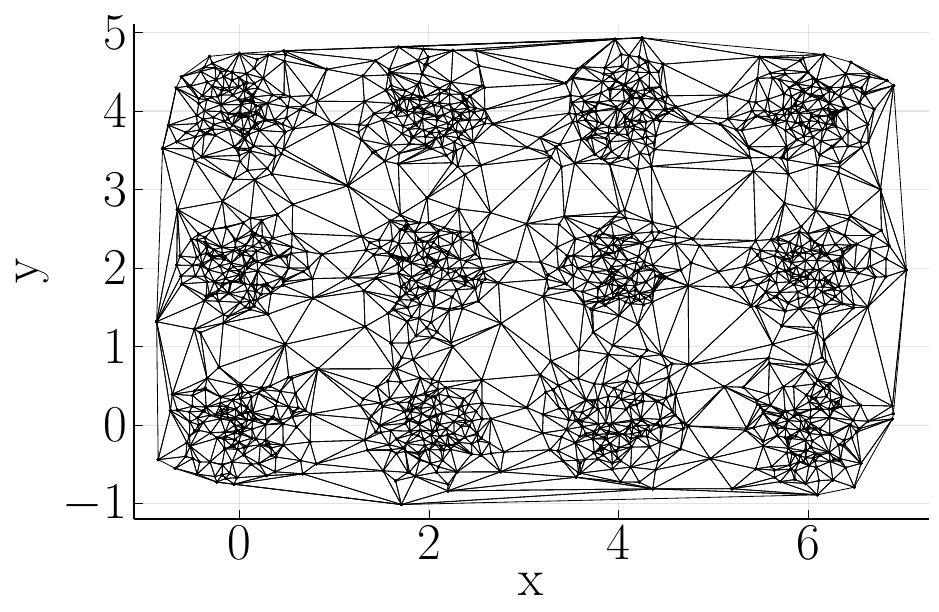}
        \end{minipage}
        \centering
        \small (a) Point clouds
        \label{fig:point_cloud_points}
    \end{minipage}
    \begin{minipage}{0.99\textwidth}
        \begin{minipage}{0.32\textwidth}
        \centering
        \safeincludegraphics[width=\linewidth]{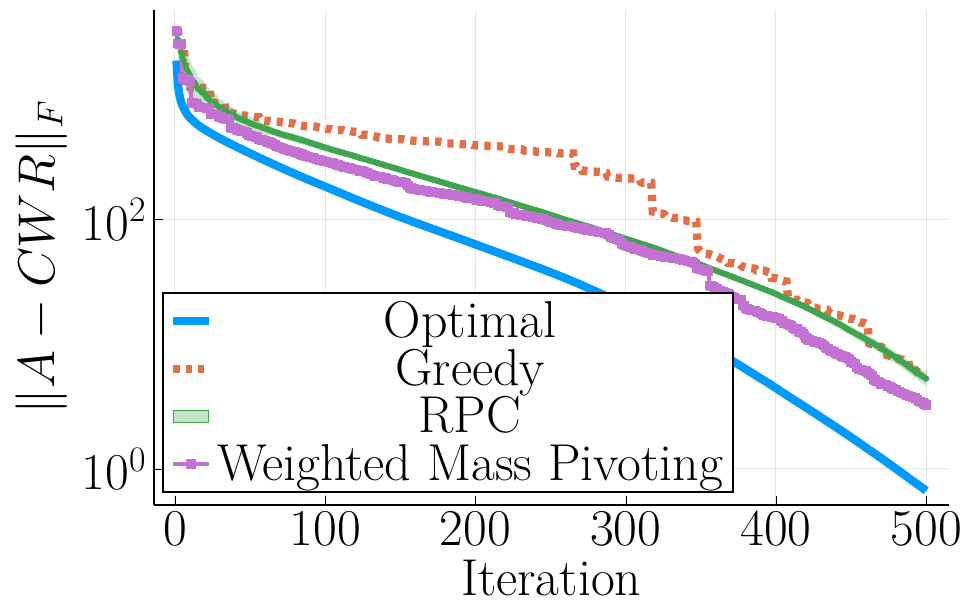}
        \end{minipage}
    % \hfill
        \begin{minipage}{0.32\textwidth}
        \centering
        \safeincludegraphics[width=\linewidth]{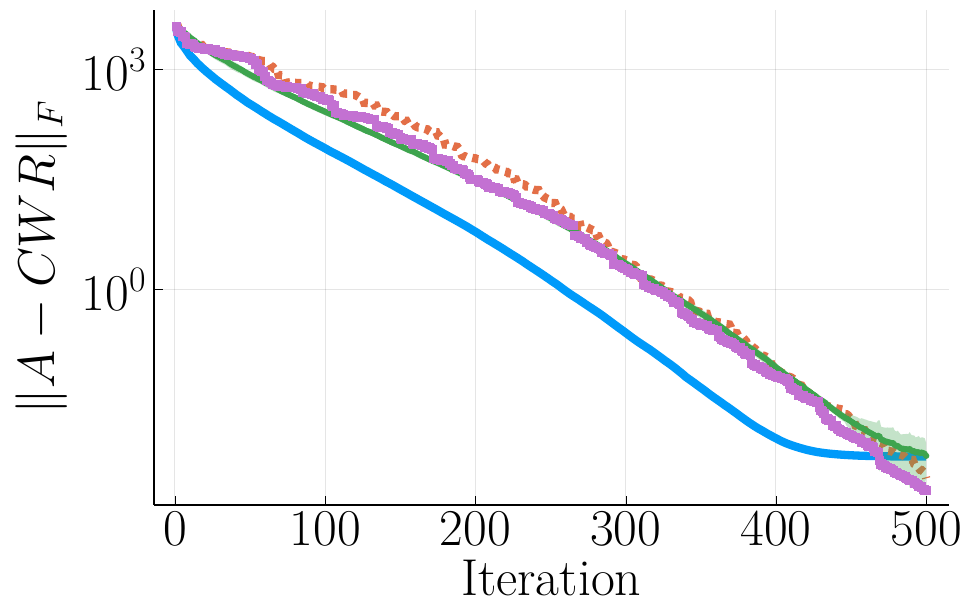}
        \end{minipage}
        \begin{minipage}{0.32\textwidth}
        \centering
        \safeincludegraphics[width=\linewidth]{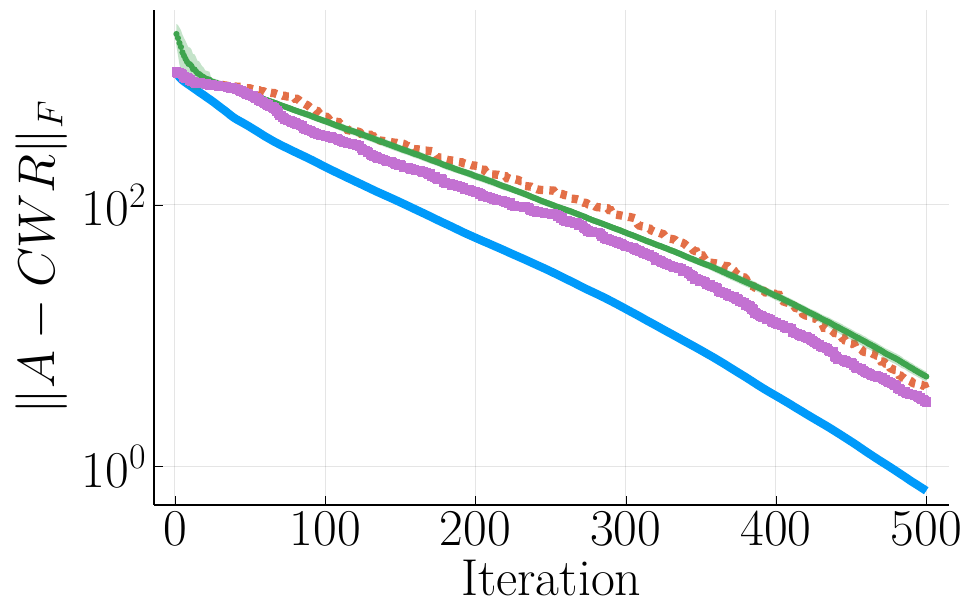}
        \end{minipage}
        \centering
        \small (b) ACA residuals and optimal residual given by SVD
        \label{fig:point_cloud_residuals}
    \end{minipage}
    \caption{Point configurations and Frobenius residuals for greedy
    diagonal pivoting, weighted-mass pivoting, and RPC. RPC reports the mean
    of $100$ runs with one standard deviation shaded. All experiments use
    $\ell=5$ and $\epsilon=10$.  }
    \label{fig:point_groups_aca}
\end{figure}

Figure \ref{fig:point_groups_aca} shows that weighted-mass pivoting gives the smallest or near smallest reported residual in all three configurations. Its advantage is most pronounced for the heterogeneous four-cluster example and the twelve-cluster grid. These configurations contain several distinct groups and therefore offer repeated opportunities to select pivots representing locally correlated residual mass. RPC also improves substantially over greedy pivoting in these two examples.

In the three-cluster experiment, we see a smaller separation among the methods. Each compact cluster appears to supply only a few  influential pivots; after these modes are removed, the three methods become more similar. In the twelve-cluster experiment, greedy pivoting eventually catches RPC, whereas weighted-mass pivoting maintains the smallest  residual. The weighted-mass pivoting strategy can likely avoid this ``catch-up'' phenomenon because it simplifies to strongly resemble greedy pivoting once all highly correlated pivots have been used. 

\subsubsection{Runtime}

The implementation uses \lstinline{CUDA.jl}
\cite{besard2018juliagpu}. On a $2500\times2500$ test matrix, an AMD Ryzen 9
7950X system with an NVIDIA GeForce RTX 4070 Ti and 32 GB of memory took,
on average, $8.167$ seconds for greedy pivoting and $8.634$ seconds for the
parallel weighted-mass implementation, measured using Julia's
\lstinline{@btime}. Although far from a scientific comparison of the algorithms' runtimes, this provides evidence that, with parallelized hardware, the additional weighted-score
calculation incurs only a modest wall-clock overhead.

Taken together, the experiments support the mechanism developed by the
exterior-algebraic analysis: a useful pivot need not be the largest residual
entry, but may instead be an index whose weighted direction is shared by many
residual rows and columns. The results do not provide a general convergence
guarantee for weighted-mass pivoting, and the score in
\eqref{eq:weighted-mass-score} remains a problem-dependent proxy for the
exact vector geometry.

\section{Conclusion}

We developed an exterior-algebraic representation of the adaptive cross-approximation residual. The classical determinant ratio becomes a weighted inner product between residual blades divided by a normalized pivot volume. This formulation separates spectral magnitudes, angular alignment, and pivot conditioning, and it recovers familiar singular-value estimates while revealing information that those estimates suppress. In particular, dependence among leading singular-vector coordinates explains how one pivot can annihilate, or nearly annihilate, rows and columns beyond the selected pair.

The rank-one specialization explicitly controls the residual by weighted sine and cosine factors. The spiral-kernel experiment shows how greedy largest-entry pivoting can ignore a broad correlated block, and the radial-basis Galerkin experiments demonstrate that a local weighted-mass proxy can exploit this structure without explicitly computing an SVD. The resulting rule is not universally optimal, but it provides a concrete example of how the geometric analysis can inform an implementable pivot strategy.

Several questions remain. A full convergence analysis of weighted-mass pivoting would require conditions connecting spatial proximity and the energy-inner-product geometry throughout the ACA iteration. Extending the blade formulation to infinite-dimensional feature maps could make the framework more natural for kernel operators and function-space discretizations. Finally, the connection with hierarchical matrices and blockwise low-rank structure \cite{Borm_IntroductionHierarchicalMatrices2003} suggests that exterior geometry may be useful for selecting not only individual pivots but also admissible blocks and local bases.

\paragraph{Code Availability}
All supporting numerical data was generated via original Julia code, which is freely available at \url{https://github.com/trevorgloe/WeightedMassGalerkin}.

\printbibliography

\appendix

\section{Weighted cosine and singular-vector coherence}\label{apx:cosine_vs_coherence}

This appendix records several rank-one observations about the average weighted cosine. They help explain why singular-vector coherence affects uniform and cosine-weighted pivot sampling. The results are descriptive rather than a recommendation to compute an SVD solely for pivot selection; in the supplied numerical tests, direct sampling proportional to weighted cosine was unstable late in the iteration and did not consistently improve on greedy pivoting.

Assume for this appendix that $\Sigma=\operatorname{diag}(\sigma_1,\ldots,\sigma_n)$ has strictly positive entries. For $U,V\in O(n)$, with rows $u_i$ and $v_j$, define
\[
C_\Sigma(U,V)_{ij}
=
\cos\angle_\sigma(u_i,v_j)
=
\frac{u_i^\top\Sigma v_j}
{\sqrt{u_i^\top\Sigma u_i}\sqrt{v_j^\top\Sigma v_j}}.
\]
Let
\(
\gamma
=
\frac{\|\sigma\|_1^2}{\|\sigma\|_2^2}.
\)
This is the effective spectral dimension introduced in Section \ref{sec:single_bound}.

\begin{prop} \label{prop:cos-identity-frame}
For every $U\in O(n)$,
\(
\|C_\Sigma(\mathbb{I},U)\|_F^2=n.
\)
\end{prop}

\begin{proof}
For the $i$th standard basis vector and row $u_j$,
\[
\cos^2\angle_\sigma(e_i,u_j)
=
\frac{\sigma_i(u_j)_i^2}{\|u_j\|_\sigma^2}.
\]
Summing first over $i$ gives $1$ for every $j$, and summing over $j$ gives $n$.
\end{proof}

\begin{prop} \label{prop:avg_cos_hadamard}
Suppose a Hadamard matrix of order $n$ exists, and let $H$ denote its orthogonal scaling, so every entry has magnitude $n^{-1/2}$. Then
\[
\|C_\Sigma(H,H)\|_F^2
=
\frac{n^2}{\gamma}
=
n^2\frac{\|\sigma\|_2^2}{\|\sigma\|_1^2}.
\]
\end{prop}

\begin{proof}
Every row $h_i$ satisfies
$\|h_i\|_\sigma^2=\|\sigma\|_1/n$. Therefore
\(
\|C_\Sigma(H,H)\|_F^2
=
\frac{n^2}{\|\sigma\|_1^2}
\sum_{i,j}(h_i^\top\Sigma h_j)^2.
\)
The sum is $\|H\Sigma H^\top\|_F^2=\|\Sigma\|_F^2=\|\sigma\|_2^2$ by orthogonal invariance of the Frobenius norm.
\end{proof}

\begin{prop}\label{prop:avg_cos_incoh}
Suppose $U$ and $V$ are $\mu$-incoherent in the entrywise sense,
\[
\max_{i,q}|U_{iq}|\leq\mu,
\qquad
\max_{j,q}|V_{jq}|\leq\mu.
\]
Then
\[
\|C_\Sigma(U,V)\|_F^2
\geq
\frac{1}{\mu^4\gamma}
=
\frac{\|\sigma\|_2^2}{\mu^4\|\sigma\|_1^2}.
\]
\end{prop}

\begin{proof}
The weighted row norms satisfy
$\|u_i\|_\sigma^2\leq\mu^2\|\sigma\|_1$ and likewise for $v_j$. Hence
\[
\|C_\Sigma(U,V)\|_F^2
\geq
\frac{1}{\mu^4\|\sigma\|_1^2}
\sum_{i,j}(u_i^\top\Sigma v_j)^2.
\]
The final sum is
$\|U\Sigma V^\top\|_F^2=\|\Sigma\|_F^2=\|\sigma\|_2^2$.
\end{proof}

For a single orthogonal frame, define
\(
f(U)=\|C_\Sigma(U,U)\|_F^2.
\)
The identity frame attains the smallest possible value, while an orthogonally scaled Hadamard frame is stationary.

\begin{theorem}\label{thm:min_max_weighted_cos_sym}
For every $U\in O(n)$,
\(
f(U)\geq n,
\)
and equality is attained at $U=I$. If an orthogonally scaled Hadamard matrix $H$ exists, then $H$ is a stationary point of $f$ on $O(n)$.
\end{theorem}

\begin{proof}
The diagonal entries of $C_\Sigma(U,U)$ are all one, so $f(U)\geq n$. At $U=I$, the off-diagonal entries vanish and $f(I)=n$.

For stationarity, set
\(
B=U\Sigma U^\top,
 d_i=B_{ii}.
\) 
Then
\(
f(U)=\sum_{i,j}\frac{B_{ij}^2}{d_id_j}.
\)
A tangent perturbation can be written $\dot U=KU$ with $K^\top=-K$, which gives
$\dot B=KB-BK$. At $U=H$, every diagonal entry equals
$c=\|\sigma\|_1/n$, and every diagonal entry of $B^2=H\Sigma^2H^\top$ equals
$r=\|\sigma\|_2^2/n$. Differentiating at this point yields
\[
Df_{\dot U}
=
\frac{2}{c^2}\inner{B,\dot B}_F
-
\frac{2}{c^3}\sum_i(B^2)_{ii}\dot B_{ii}.
\]
The first term vanishes because
\(\inner{B,KB-BK}_F=\operatorname{tr}(BKB-B^2K)=0. \)
The second vanishes because $(B^2)_{ii}=r$ is constant and $\sum_i\dot B_{ii}=\operatorname{tr}(KB-BK)=0$. Hence every tangent directional derivative is zero at $H$.
\end{proof}

Theorem \ref{thm:min_max_weighted_cos_sym} does not establish that the Hadamard frame is a global maximizer. The supplied numerical experiments suggest the conjecture
\(
\|C_\Sigma(U,V)\|_F^2
\leq
\frac{n^2}{\gamma}
\)
with equality at a common Hadamard frame, but a proof is not provided here.

\subsection{Sampling proportional to weighted cosine}

Suppose a rank-one pivot $(i,j)$ is sampled with probability
\begin{equation}\label{eq:cosine-sampling-probability}
\prob\{(i,j)\}
=
\frac{C_\Sigma(U,V)_{ij}^2}
{\|C_\Sigma(U,V)\|_F^2}.
\end{equation}
Zero-cosine pairs receive zero probability. The cosine denominator in Theorem \ref{thm:single_iter_cos_bound} then cancels in expectation.

\begin{prop}\label{prop:cosine-sampling-entry}
Let $E^{(i,j)}$ be the residual after the rank-one pivot $(i,j)$. Under
\eqref{eq:cosine-sampling-probability},
\begin{align*}
\Exp\left|E^{(i,j)}_{\ell p}\right|^2
&\leq
\frac{\|u_\ell\|_\sigma^2\|v_p\|_\sigma^2}
{\|C_\Sigma(U,V)\|_F^2}
\left(\sum_i\sin^2\angle_\sigma(u_\ell,u_i)\right)
\left(\sum_j\sin^2\angle_\sigma(v_p,v_j)\right).
\end{align*}
\end{prop}

\begin{proof}
Square the bound in Theorem \ref{thm:single_iter_cos_bound}, multiply by the probability in \eqref{eq:cosine-sampling-probability}, and sum over $(i,j)$. The cosine factors cancel, and the remaining double sum separates.
\end{proof}

This yields a coarse incoherence-dependent Frobenius estimate.

\begin{corollary}\label{corr:cosine-sampling-incoherent}
Suppose $U$ and $V$ are $\mu$-incoherent and let
$\eta=\sqrt n\,\mu$. Under cosine-squared sampling,
\(\Exp\|E^{(i,j)}\|_F^2\leq\big(\eta^4\gamma-1\big)^2\|A\|_F^2.
\)
\end{corollary}

\begin{proof}
Write \(f(U,V)=\|C_\Sigma(U,V)\|_F^2, f(U)=\|C_\Sigma(U,U)\|_F^2. \)
Summing Proposition \ref{prop:cosine-sampling-entry} over $\ell,p$ and using $\|u_\ell\|_\sigma^2,\|v_p\|_\sigma^2\leq\mu^2\|\sigma\|_1$ gives
\[\Exp\|E^{(i,j)}\|_F^2\leq\frac{\mu^4\|\sigma\|_1^2}{f(U,V)}\big(n^2-f(U)\big)\big(n^2-f(V)\big).
\]
By Proposition \ref{prop:avg_cos_incoh},
\(f(U,V),f(U),f(V)\geq\frac{n^2}{\eta^4\gamma}.
\)
Substitution yields
\[
\Exp\|E^{(i,j)}\|_F^2
\leq
\eta^8\gamma\|\sigma\|_1^2
\left(1-\frac{1}{\eta^4\gamma}\right)^2.
\]
Since $\|A\|_F^2=\|\sigma\|_2^2$ and
$\gamma=\|\sigma\|_1^2/\|\sigma\|_2^2$, the right-hand side equals
$\big(\eta^4\gamma-1\big)^2\|A\|_F^2$.
\end{proof}

The bound is usually loose and can exceed the trivial scale $\|A\|_F^2$. It nevertheless separates two effects: $\eta$ measures singular-vector incoherence and $\gamma$ measures effective spectral dimension. The one-step estimate cannot be iterated without controlling how both quantities change after each ACA update.

\section{Nearest-neighbor preprocessing}
\label{apx:nearest-neighbor}
\begin{algorithm}[t]
\caption{Insertion search for the $\ell$ nearest centers to one query point}
\label{alg:l_closest_gpu}
\begin{algorithmic}[1]
\Require Centers $X=\{x_j\}_{j=1}^n$, query index $i$, neighborhood size $\ell$
\Ensure Sorted neighbor indices $C[1{:}\ell]$ and distances $D[1{:}\ell]$
\State $C[1{:}\ell]\gets-1$, $D[1{:}\ell]\gets+\infty$
\For{$j=1,\ldots,n$}
    \State $d\gets\|x_i-x_j\|_2$
    \For{$t=1,\ldots,\ell$}
        \If{$d<D[t]$}
            \For{$s=\ell,\ell-1,\ldots,t+1$}
                \State $C[s]\gets C[s-1]$
                \State $D[s]\gets D[s-1]$
            \EndFor
            \State $C[t]\gets j$, $D[t]\gets d$
            \State \algorithmicbreak
        \EndIf
    \EndFor
\EndFor
\State \Return $C,D$
\end{algorithmic}
\end{algorithm}
Weighted-mass pivoting requires the neighborhood sets
$\mathcal N_i$ containing the $\ell$ nearest centers to each $x_i$. These
neighborhoods depend only on the point geometry and are therefore computed
once before the ACA iteration. The implementation used in the experiments
assigns one query center to each GPU thread. Algorithm
\ref{alg:l_closest_gpu} describes the insertion search performed by a single
thread. 

For one query center, Algorithm \ref{alg:l_closest_gpu} requires
$\O(nl)$ arithmetic and $\O(\ell)$ temporary storage. Applying the search to all $n$ centers therefore requires $\O(n^2l)$ total work and produces an $\O(nl)$ neighborhood table. The $n$ searches are independent, however, and can be executed in parallel. If the neighborhood convention excludes the query center itself, the case $j=i$ is simply omitted from the outer loop.  The insertion search is convenient for the moderate neighborhood sizes used in the experiments. For larger point sets or larger neighborhoods, the same weighted-mass algorithm can instead use a spatial data structure or an approximate nearest-neighbor method; the pivot rule itself is independent of how the neighborhoods are constructed.

\end{document}